\documentclass[10pt]{article}
\usepackage[T1]{fontenc}
\usepackage{amsmath,amsfonts,amsthm,mathrsfs,amssymb,cite}
\usepackage{graphicx}
\usepackage{subfigure}
\usepackage{placeins}
\usepackage{color}
\usepackage{indentfirst}
\usepackage{hyperref}
\numberwithin{equation}{section}
\newcommand{\R}{{\mathbb R}}
\newcommand{\Z}{{\mathbb Z}}

\newcommand{\C}{{\mathbb C}}

\newcommand{\be}{\begin{eqnarray}}
\newcommand{\ben}{\begin{eqnarray*}}
\newcommand{\en}{\end{eqnarray}}
\newcommand{\enn}{\end{eqnarray*}}

\newcommand{\pa}{\partial}

\newcommand{\ov}{\overline}

\newcommand{\divv}{{\rm div\,}}
\newcommand{\real}{{\rm Re\,}}

\newcommand{\G}{\Gamma}

\newcommand{\om}{\omega}

\newtheorem{theorem}{Theorem}[section]
\newtheorem{lemma}[theorem]{Lemma}

\newtheorem{definition}[theorem]{Definition}
\newtheorem{remark}[theorem]{Remark}
\newtheorem{proposition}[theorem]{Proposition}

\definecolor{rot}{rgb}{0.000,0.000,0.000}
\definecolor{rot1}{rgb}{0.000,0.000,0.000}
\begin{document}
\renewcommand{\theequation}{\arabic{section}.\arabic{equation}}
\begin{titlepage}
\title{\bf Multi-frequency iterative methods for the inverse medium scattering problems in elasticity}

\author{Gang Bao\thanks{School of Mathematical Sciences, Zhejiang University, Hangzhou 310027, China. Email: {\tt baog@zju.edu.cn}}\;,
Fang Zeng\thanks{College of Mathematics and Statistics, Chongqing University, Chongqing 400044, China. Email:{\tt fzeng@cqu.edu.cn}}\;,
Tao Yin\thanks{Department of Computing \& Mathematical Sciences, California Institute of Technology, 1200 East California Blvd., CA 91125, United States. Email:{\tt taoyin89@caltech.edu}}
}
\date{}
\end{titlepage}
\maketitle
\vspace{.2in}

\begin{abstract}
This paper concerns the reconstruction of multiple elastic parameters (Lam\'e parameters and density) of an inhomogeneous medium embedded in an infinite homogeneous isotropic background in $\R^2$. The direct scattering problem is reduced to an equivalent system on a bounded domain by introducing an exact transparent boundary condition and the wellposedness of the corresponding variational problem is established. The Fr\'{e}chet differentiability of the near-field scattering map is studied with respect to the elastic parameters. Based on the multi-frequency measurement data and its phaseless term, two Landweber iterative algorithms are developed for the reconstruction of the multiple elastic parameters. Numerical examples, indicating that plane pressure incident wave is a better choice, are presented to show the validity and accuracy of our methods.

\vspace{.2in}
Keywords: Elastic wave, inverse medium problem, iterative method, multi-frequency
\end{abstract}

\section{Introduction}
\label{sec:1}

Time-harmonic elastic scattering problems play important roles in many fields of applications and the linear elasticity theory provides an essential tool for analysis and design of mechanic systems and engineering structures (\cite{Ammari2015,Ciarlet}). In this paper, we consider several inhomogeneous isotropic elastic bodies embedded in an infinite homogeneous isotropic background medium in $\R^2$. Denote by $\lambda$, $\mu$ the Lam\'e parameters with $\mu>0$, $\lambda>0$, and by $\rho$ the density of the elastic medium. Suppose that $\lambda=\lambda_0(1+q_\lambda)$, $\mu=\mu_0(1+q_\mu)$ and $\rho=\rho_0(1+q_\rho)$ where $\lambda_0$, $\mu_0$ and $\rho_0$ are constants representing the Lam\'e parameters and density of the background elastic medium. Denote $B_R:=\{x\in\R^2:|x|<R\}$ and $\Gamma_R:=\pa B_R$. Set $q:=(q_\lambda,q_\mu,q_\rho)^\top$. Throughout, we make the following assumption:
\begin{description}
\item[Assumption:] there exists some $R>0$ and constants $C_1,C_2>0$ such that
    \ben
    \mbox{supp}\,\{q\}\subset B_R, q\in\mathcal{K}:=\{q\in L^\infty(B_R)^3: -1<C_1\le q_\lambda,q_\mu,q_\rho\le C_2<\infty\}.
    \enn
\end{description}

Let $u^{in}$ be a plane incident field satisfying
\be
\label{Navier-inc}
\nabla\cdot \sigma_{0,0}(u^{in})+   \rho_0\omega^2u^{in} = 0\quad\mbox{in}\quad \R^2,
\en
where $\omega>0$ is the frequency and the stress tensor $\sigma_{q_\lambda,q_\mu}(u)$ is defined as
\ben
\sigma_{q_\lambda,q_\mu}(u):=\lambda_0(1+q_\lambda)(\mbox{div}\,u){\bf I}+ 2\mu_0(1+q_\mu)\mathcal{E}(u),\quad \mathcal{E}(u):=\frac{1}{2}\left(\nabla\,u+ \nabla\,u^\top\right),
\enn
and ${\bf I}$ stands for the $2\times2$ identity matrix. The elliptic equation (\ref{Navier-inc}) can be restated as
\ben
\Delta_{\lambda_0,\mu_0}^*u^{in}+   \rho_0\omega^2u^{in} = 0\quad\mbox{in}\quad \R^2,
\enn
where the Lam\'e operator $\Delta_{\lambda,\mu}^*$ is defined as
\ben
\Delta_{\lambda,\mu}^* = \mu\,\mbox{div}\,\mbox{grad} + (\lambda + \mu)\,\mbox{grad}\, \mbox{div}.
\enn
In this paper, the incident wave is allowed to be either a plane shear wave taking the form
\ben
u^{in}=u_s^{in}:=d^\perp e^{ik_sx\cdot d},\quad d=(\cos\theta^{in},\sin\theta^{in})^\top\in \textcolor{rot1}{\Gamma_1}, \textcolor{rot1}{d^\perp=(-\sin\theta^{in},\cos\theta^{in})^\top,}
\enn
or a plane pressure wave taking the form
\ben
u^{in}=u_p^{in}:=d e^{ik_px\cdot d},\quad d\in \textcolor{rot1}{\Gamma_1},
\enn
where
\ben
k_s = \omega\sqrt{\frac{\rho_0}{\mu_0}},\quad k_p = \omega \sqrt{\frac{\rho_0}{\lambda_0 + 2\mu_0}},
\enn
are the wave numbers of pressure wave and shear wave, respectively and $d$ and $\theta^{in}$ are referred as the direction and angle of the incidence, respectively.

The total displacement field $u=(u_1,u_2)^\top$ can be modeled by the reduced Navier equation
\be
\label{Navier-tot}
\nabla\cdot \sigma_{q_\lambda,q_\mu}(u)+\omega^2\rho u=0\quad\mbox{in}\quad \R^2.
\en
Since the background medium is unbounded, an appropriate radiation condition at infinity must be imposed on the scattered field $u^{sc}:=u-u^{in}$ to ensure well-posedness of the scattering problem. The scattered field in $\R^2\backslash\ov{B_R}$ can be decomposed into the sum of the compressional (longitudinal) part ${u}^{sc}_p$ and the shear (transversal) part ${u}^{sc}_s$ as follows :
\be
\label{decomposition}
u^{sc}={u}^{sc}_p+{u}^{sc}_s,\quad
{u}^{sc}_p=-\frac{1}{k_p^2}\,\mbox{grad}\,\mbox{div}\;{u}^{sc},\quad {u}^{sc}_s=\frac{1}{k_s^2}\,\overrightarrow{\mbox{curl}}\,\mbox{curl}\;{u}^{sc},
\en
where the two-dimensional operators \mbox{curl} and $\overrightarrow{\mbox{curl}}$ are defined respectively by
\ben
\mbox{curl}\,v=\partial_1 v_2-\partial_2 v_1,\quad v=(v_1,v_2)^\top,\qquad \overrightarrow{\mbox{curl}}\; f:=(\partial_2f, -\partial_1f)^\top.
\enn
It then follows from the decompositions in (\ref{decomposition}) that
\ben
(\Delta+k_\alpha^2)\, u_\alpha^{sc}=0,\qquad \alpha=p,s,\qquad \divv u_s^{sc}=0,\quad\textcolor{rot1}{\mbox{curl}}\,u_p^{sc}=0.
\enn
The scattered field is required to satisfy the Kupradze radiation condition (see e.g. \cite{Kupradze})
\be
\label{RadiationCond}
\lim_{r \to \infty} r^{\frac{1}{2}}\left(\frac{\partial u^{sc}_t}{\partial r}-ik_tu_t^{sc}\right) = 0,\quad r=|x|,\quad t=p,s,
\en
uniformly with respect to all $\hat{x}=x/|x|\in \textcolor{rot1}{\Gamma_1}$.

Given the incident field $u^{in}$, the direct problem is to determine the scattered field $u^{sc}$ for the known elastic parameters $\lambda,\mu,\rho$. In practice, the original boundary value problem for the Lam\'e system can be reduced to an equivalent system on a bounded domain via introducing an exact transparent boundary condition (TBC) on an artificial boundary enclosing the inhomogeneous bodies. The TBC can be formulated by the so-called Dirichlet-to-Neumann (DtN) map taking in the form of a Fourier series (\cite{BHSY,GK1990,Li2016,LY17}). Based on the properties of the DtN map and Fredholm alternative theorem, the uniqueness and existence of weak solutions of the equivalent system can be derived(Section \ref{sec:2}, Theorem \ref{wellposedness}).

The main purpose of this paper is to study numerical algorithm for the inverse medium problem in the elastic scattering, that is, to determine the unknown elastic parameters $q_\lambda, q_\mu, q_\rho$ from the measurements of near-field data $u|_{\G_R}$, given the incident field $u^{in}$. For static case, i.e., $\omega=0$, uniqueness of the inverse medium problem has been investigated under appropriate assumptions on the elastic parameters in \cite{ANS91,BFV14,ER02,IY15,NU93,NU95,NU941}. For the time-harmonic case, Beretta et al (\cite{BHFV17}) proves uniqueness when the Lam\'e parameters and the density are assumed to be piecewise constant on a given domain partition. For the mathematical analysis of the stability for the inverse medium problems in elasticity, we refer to \cite{ACMR14,BFV14,I90,NU93,NU941,NU95}.

Recently, it has been realized that the use of multi-frequency data is an effective approach to overcome the major difficulties associated with the inverse medium problems: the ill-posedness and the presence of many local minima. Based on the multi-frequency measurements and the Fr\'echet derivative of the solution operator, a stable recursive linearization method is proposed in \cite{BL04} for the inverse medium problems in acoustics, see also in \cite{BL052,BL051,BL07,BL09}. The idea to use multi-frequency data has also been widely developed in proving uniqueness and increasing stability for inverse source problems in acoustics, elastodynamics and electromagnetics (\cite{BHKY18,BLLT,BLZ,BLT,BT,CIL16,HLLZ,IL18,LGY17,LY17}). It still remains open whether or not the multi-frequency measurements uniquely determine the Lam\'e parameters and the density? This paper is designed to study the capability of iterative methods for inverse medium problems in elasticity using multi-frequency measurements. Compared with the acoustic and electromagnetic case, the elasticity problem appears to be more complicated because of the coexistence of pressure and shear waves that propagate at different speeds and the aim to reconstruct multiple parameters. In particular, if the \textcolor{rot1}{Lam\'e} parameters are constants, the inverse medium problem in elasticity is consistent with that in acoustics, see Section \ref{sec:4}. Relying on the variational arguments for the direct problems, we derive the Fr\'echet derivative of the solution operator with respect to the elastic parameters and investigate the adjoint of the Fr\'echet derivative. Then we employ the Landweber iterative method based on the multi-frequency measurements to find the unknown elastic parameters. At each iteration step, the forward problem and an adjoint one need to be solved and the correctness of the parameters needs to be evaluated.

The outline of the paper is as follows. In Section \ref{sec:2}, we derive the well-posedness of the direct scattering problem using variational approach and investigate the Fr\'echet differentiability of the near-field scattering map. We develop the Landweber iterative methods for solving the inverse medium problem in Section \ref{sec:3}. In Section \ref{sec:4}, we give a brief discussion about a special case that $q_\lambda=q_\mu=0$. Numerical examples are presented in Section \ref{sec:5}.

\section{Direct scattering problem}
\label{sec:2}

In this section, we discuss the \textcolor{rot1}{well-posedness} of the direct elastic scattering problem and investigate the Fr\'echet differentiability of the near-field scattering map.

\subsection{Variational formulation}
\label{sec:2.1}

We first reduce the original scattering problem described in section \ref{sec:1} on a bounded domain via introducing a TBC on an artificial boundary $\Gamma_R$ enclosing the inhomogeneity inside. The TBC is formulated by the so-called DtN map defined as follows.
\begin{definition}
For any $w\in H^{1/2}(\G_R)^2$, The DtN map $\mathcal{B}$ applied to $w$ is defined as $T_{\lambda_0,\mu_0}v^{sc}|_{\G_R}$, where $v^{sc}$ satisfies
\be
\label{def1}
\Delta^{*}_{\lambda_0,\mu_0}v^{sc} +   \rho_0\omega^2v^{sc} &=& 0\quad\mbox{in}\quad \R^2\backslash\ov{B_R},\\
\label{def2}
v^{sc} &=& w \quad\mbox{on}\quad \G_R,
\en
and the Kupradze radiation condition. Here, $T_{\lambda_0,\mu_0}$ is the traction operator defined by
\ben
T_{\lambda_0,\mu_0}u:=\nu\cdot\sigma_{0,0}(u)=2\mu_0 \, \partial_{\nu} u + \lambda_0 \,
\nu \, \divv u
-\mu_0 \nu^{\perp}\,\rm{curl}\,u,
\enn
where \textcolor{rot1}{$\nu=(\nu_1,\nu_2)^\top$} denotes the exterior unit normal vector to $\G_R$ \textcolor{rot1}{and the corresponding tangential vector is given by $\nu^{\perp}:=(-\nu_2,\nu_1)^\top$}.
\end{definition}

The DtN map $\mathcal{B}$ is well-defined since the Dirichlet-kind boundary value problem (\ref{def1})-(\ref{def2}) is uniquely solvable in $H^1_{loc}(\R^2\backslash\ov{B_R})^2$, see Corollary 2.3 in \cite{BHSY}. For all $n\in\Z$, denote
\ben
\alpha_n(t_\xi):=\frac{{H_n^{(1)}}'(t_\xi)}{H_n^{(1)}(t_\xi)},\quad \beta_n(t_\xi):=\frac{{H_n^{(1)}}''(t_\xi)}{H_n^{(1)}(t_\xi)},\quad \xi=p,s,
\enn
where
\ben
t_\xi=k_\xi R.
\enn
Following the procedure described in \cite{BHSY}, it can be derived that
\be
\label{DtN}
\mathcal{B}w=\sum_{n\in\Z}\frac{1}{2\pi R}M_\theta^\top W_n \int_0^{2\pi} M_\phi w(R,\phi) e^{in(\theta-\phi)}d\phi,
\en
where the matrix $M_\theta$ dependent on the angle $\theta\in[0, 2\pi)$ is defined as
\ben
M_\theta:=\begin{bmatrix}
\cos\theta & \sin\theta \\
-\sin\theta & \cos\theta
\end{bmatrix}
\enn
and the coefficient matrix $W_n$ is given by
\be
\label{MatrixW}
W_n=B_nA_n^{-1},
\en
with
\ben
&&A_n:=\begin{bmatrix}
t_p\alpha_n(t_p) & in \\
in & -t_s\alpha_n(t_s)
\end{bmatrix},\\
&&B_n:= \begin{bmatrix}
2\mu_0 t_p^2\beta_n(t_p) -\lambda_0 t_p^2 & 2i\mu_0 n\left( t_s\alpha_n(t_s)-1\right) \\
2i\mu_0 n\left(t_p\alpha_n(t_p)-1\right) & -2\mu_0 t_s^2\beta_n(t_s) -\mu_0 t_s^2
\end{bmatrix}.
\enn
\textcolor{rot1}{For the invertibility of matrix $A_n$, we refer to Lemma 2.11 in \cite{BHSY}. Denote by $\Lambda_n$ the determinant of $A_n$.}

\begin{lemma}
For all $\varphi\in (H^{1/2}(\Gamma_R))^2$, the DtN mapping $\mathcal{B}$ can be expressed equivalently as
\be
\label{expressionDtN2}
\mathcal{B}\varphi:=\sum_{n\in\Z}\frac{1}{2\pi R}M_\theta^\top\,[W_n]^\top\int_0^{2\pi} M_\phi\,\varphi\,e^{in(\phi-\theta)}d\phi,
\en
In addition, it holds that
\be
\label{adjointDtN}
\mathcal{B}^*\varphi=\ov{\mathcal{B}\,\ov{\varphi}}.
\en
\end{lemma}
\begin{proof}
We first prove (\ref{adjointDtN}). We can derive that for all $\varphi,\psi\in (H^{1/2}(\Gamma_R))^2$,
\ben
\langle\mathcal{B}^*\varphi,\psi\rangle_{\Gamma_R} &=& \langle\varphi,\mathcal{B}\psi\rangle_{\Gamma_R} \\
&=& \int_0^{2\pi}\sum_{n\in\Z}\frac{1}{2\pi} \varphi^\top M_\theta^\top\,\ov{W_n}\int_0^{2\pi} M_\phi\,\ov{\psi}\,e^{in(\phi-\theta)}d\phi d\theta \\
&=& \int_0^{2\pi}\sum_{n\in\Z}\frac{1}{2\pi} \ov{\psi}^\top M_\phi^\top\,\ov{[W_n]^\top}\int_0^{2\pi} M_\theta\,\varphi\,e^{in(\phi-\theta)}d\theta d\phi \\
&=& \int_0^{2\pi}\sum_{n\in\Z}\frac{1}{2\pi} \ov{\psi}^\top \ov{\left\{M_\phi^\top\,[W_n]^\top\int_0^{2\pi} M_\theta\,\ov{\varphi}\,e^{in(\theta-\phi)}d\theta\right\}} d\phi \\
&=& \langle\ov{\mathcal{B}\ov{\varphi}},\psi\rangle_{\Gamma_R}
\enn
where $\langle\cdot,\cdot\rangle_{\Gamma_R}$ is the $L^2$ duality pairing between $(H^{-1/2}(\Gamma_R))^2$ and $(H^{1/2}(\Gamma_R))^2$. It remains to prove (\ref{expressionDtN2}).
It follows from (\ref{DtN}) that
\ben
\mathcal{B}w &=& \sum_{n\in\Z}\frac{1}{2\pi R}M_\theta^\top W_n \int_0^{2\pi} M_\phi w(R,\phi) e^{in(\theta-\phi)}d\phi\\
&=& \sum_{n\in\Z}\frac{1}{2\pi R}M_\theta^\top W_{-n} \int_0^{2\pi} M_\phi w(R,\phi) e^{in(\phi-\theta)}d\phi
\enn
where
\ben
W_{-n}=B_{-n}A_{-n}^{-1},
\enn
with
\ben
&&A_{-n}:=\begin{bmatrix}
t_p\alpha_n(t_p) & -in \\
-in & -t_s\alpha_n(t_s)
\end{bmatrix},\\
&&B_{-n}:= \begin{bmatrix}
2\mu_0 t_p^2\beta_n(t_p) -\lambda_0 t_p^2 & -2i\mu_0 n\left( t_s\alpha_n(t_s)-1\right) \\
-2i\mu_0 n\left(t_p\alpha_n(t_p)-1\right) & -2\mu_0 t_s^2\beta_n(t_s) -\mu_0 t_s^2
\end{bmatrix}.
\enn
\textcolor{rot1}{It follows from the properties of Hankel function that
\ben
{H_n^{(1)}}''(z)=\left(\frac{n^2}{z^2}-1\right)H_n^{(1)}(z)- \frac{1}{z} {H_n^{(1)}}'(z),
\enn
giving rise to the identities
\ben
\beta_n(t_p)=\frac{n^2}{t_p^2}-1-\frac{1}{t_p}\alpha_n(t_p), \quad \beta_n(t_s)=\frac{n^2}{t_s^2}-1-\frac{1}{t_s}\alpha_n(t_s).
\enn
From the expressions of $A_n^{-1}$ and $B_n$ we get the entries $W_n^{(i,j)}$ of $W_n$, given by (see also Lemma 2.13 in \cite{BHSY})
\ben
W_n^{(1,1)} &=& \frac{1}{R\Lambda_n}\left[ -2\mu\Lambda_n +\rho_0\omega^2R^2t_s\alpha_n(t_s)\right], \\
W_n^{(2,2)} &=& \frac{1}{R\Lambda_n}\left[-2\mu\Lambda_n +\rho_0\omega^2R^2t_p\alpha_n(t_p)\right],\\
W_n^{(1,2)} &=& \frac{1}{R\Lambda_n}\left[ -2in\mu\Lambda_n +in\rho_0\omega^2R^2\right],\\
W_n^{(2,1)} &=& -W_n^{(1,2)}.
\enn
Then we can obtain that $W_{-n}^{(1,1)}=W_n^{(1,1)}$, $W_{-n}^{(2,2)}=W_n^{(2,2)}$ and $W_{-n}^{(1,2)}=W_n^{(2,1)}$. These further imply that $W_{-n}=[W_n]^\top$ which completes the proof.}
\end{proof}

The following property of the coefficient matrix $W_n$ can be obtained directly from Lemma 2.13 in \cite{BHSY}.
\begin{lemma}
\label{positive}
The matrix $\widetilde{W}_n=-(W_n+W_n^*)/2$ is positive definite for sufficiently large $|n|$.
\end{lemma}

Using the DtN map, we can impose the following TBC for the scattered field
\ben
T_{\lambda_0,\mu_0}u^{sc}=\mathcal{B}u^{sc} \quad\mbox{on}\quad\G_R.
\enn
Note that
\ben
T_{\lambda_0,\mu_0}u=\mathcal{B}u+T_{\lambda_0,\mu_0}u^{in}-\mathcal{B}u^{in}\quad\mbox{on}\quad\G_R.
\enn
Then the original scattering problem is equivalently reduced to the following nonlocal boundary value problem
\be
\label{total}
\nabla\cdot \sigma_{q_\lambda,q_\mu}(u)+\rho\omega^2 u=0&&\quad\mbox{in}\quad B_R,\\
\label{TBC}
T_{\lambda_0,\mu_0}u-\mathcal{B}u-g =0 &&\quad\mbox{on}\quad\G_R,
\en
where $g:=T_{\lambda_0,\mu_0}u^{in}-\mathcal{B}u^{in}$. Then the variational formulation of (\ref{total})-(\ref{TBC}) reads as follows: find $u=(u_1,u_2)^\top\in (H^1(B_R))^2$ such that
\be
\label{variational}
a_{q}(u,v)-\int_{\G_R} \mathcal{B}u\cdot\overline{v}\,ds=\int_{\G_R} g\cdot\overline{v}\,ds\quad\mbox{for all}\quad v=(v_1,v_2)^\top\in (H^1(B_R))^2,
\en
where the sesquilinear form $a_q(\cdot,\cdot): (H^1(B_R))^2\times (H^1(B_R))^2\rightarrow \C $ is defined by
\be
\label{sesquilinear}
a_{q}(u,v):=A_{q_\lambda}(u,v)+B_{q_\mu}(u,v)+ C_{q_\rho}(u,v)
\en
and
\ben
A_{q_\lambda}(u,v)&=&\int_{B_R}\lambda_0(1+q_\lambda)(\nabla\cdot u)(\nabla\cdot\ov{v})dx,\\ B_{q_\mu}(u,v)&=&2\int_{B_R}\mu_0(1+q_\mu)\,\mathcal{E}(u):\mathcal{E}(\ov{v})dx,\\
C_{q_\rho}(u,v)&=&-\int_{B_R}\rho_0(1+q_\rho)\omega^2\, u\cdot\ov{v}dx.
\enn
\textcolor{rot1}{The double dot notation appeared in $B_{q_\mu}$ is understood in the following way. If tensors ${\bf A}$ and ${\bf B}$ have rectangular Cartesian components $a_{ij}$ and $b_{ij}$, $i,j=1,2$, respectively, then the double contraction of ${\bf A}$ and ${\bf B}$ is
\ben
{\bf A}:{\bf B}=\sum\limits_{i = 1}^2 {\sum\limits_{j = 1}^2 {{a_{ij}}{b_{ij}}}}.
\enn}The \textcolor{rot1}{well-posedness} of the variational formulation (\ref{variational}) is a consequence of the following theorem.

\begin{theorem}
\label{wellposedness}
For any $f\in L^2(B_R)^2$ and $g\in H^{-1/2}(\G_R)^2$, there exists a unique weak solution $u\in H^1(B_R)^2$ to the boundary value problem
\be
\label{BVP1}
\nabla\cdot \sigma_{q_\lambda,q_\mu}(u)+\rho\omega^2 u=f&&\quad\mbox{in}\quad B_R,\\
\label{BVP2}
Tu-\mathcal{B}u-g =0 &&\quad\mbox{on}\quad\G_R,
\en
and
\be
\label{solestimate}
\|u\|_{H^1(B_R)^2}\le \gamma_q^{-1}(\|f\|_{L^2(B_R)^2}+\|g\|_{H^{-1/2}(\G_R)^2}),
\en
where $\gamma_q>c_0>0$ is a constant independent of $u$.
\end{theorem}
\begin{proof}
We know from Lemma \ref{positive} that $-\mbox{Re}(W_n)$ is positive definite for large $|n|$. The operator \textcolor{rot1}{$-\mathcal{B}$} can be decomposed into the sum of \textcolor{rot1}{an operator $\mathcal{B}_1$ whose real part is positive definite} and a finite rank operator $\mathcal{B}_2$ from $H^{1/2}(\G_R)^2$ to $H^{-1/2}(\G_R)^2$. Define the sesquilinear form $A(\cdot,\cdot): H^1(B_R)^2\times H^1(B_R)^2\rightarrow \C $ as
\ben
A(u,v)=a_{q}(u,v)-\int_{\G_R} \mathcal{B}u\cdot\overline{v}\,ds.
\enn
\textcolor{rot1}{Then the variational equation corresponding to (\ref{BVP1})-(\ref{BVP2}) reads:
\be
\label{variational0}
A(u,v)=\int_{\G_R} \mathcal{B}g\cdot\overline{v}\,ds-\int_{B_R}f\cdot\ov{v}\,ds, \quad\forall v\in H^1(B_R)^2.
\en}
We split the sequilinear form $A$ into the sum $A=A_1+A_2$, where
\ben
A_1(u,v)&=&\int_{B_R}\left[\lambda_0(1+q_\lambda)(\nabla\cdot u)(\nabla\cdot\ov{v}) +2\mu_0(1+q_\mu)\,\mathcal{E}(u):\mathcal{E}(\ov{v})
 +u\cdot\ov{v}\right] dx\\
&\quad& +\int_{\G_R} \mathcal{B}_1u\cdot\overline{v}\,ds,\\
A_2(u,v)&=&-\int_{B_R}\left[\rho_0(1+q_\rho)\omega^2+1\right]u\cdot\ov{v} dx+\int_{\G_R} \mathcal{B}_2u\cdot\overline{v}\,ds.
\enn
Recalling the Korn's inequality (see e.g. \cite{HKR00}), we have
\ben
\mbox{Re}\,A_1(v,v)\ge c_1\|v\|_{H^1(B_R)^2}^2\quad\mbox{for all}\quad v\in H^1(B_R)^2,
\enn
with some constant $c_1>0$. Moreover, applying the Cauchy-Schwarz inequality yields
\ben
\mbox{Re}\,A_2(v,v)\ge -c_2\|v\|_{L^2(B_R)^2}^2 +\mbox{Re}\langle\mathcal{B}_2v,v\rangle_{\Gamma_R}\quad\mbox{for all}\quad v\in H^1(B_R)^2,
\enn
for some constant $c_2>0$. From the compact imbedding $H^1(B_R)\hookrightarrow L^2(B_R)$ and the compactness of $\mathcal{B}_2$, we conclude that the sesquilinear form $A$ is strongly elliptic \textcolor{rot1}{(see Definition \ref{def:stronglyelliptic})} over $H^1(B_R)^2\times H^1(B_R)^2$. The sesquilinear form $A$ obviously generates a continuous linear operator $\mathcal{A}: H^1(B_R)^2\rightarrow (H^1(B_R)^2)'$ such that
\ben
A(u,v)=\langle\mathcal{A}u,v\rangle\quad\mbox{for all}\quad v\in H^1(B_R)^2.
\enn
Here $(H^1(B_R)^2)'$ denotes the dual space of $H^1(B_R)^2$ with respect to the duality $\langle\cdot,\cdot\rangle$ extending the $L^2$ scalar product in $L^2(B_R)^2$. We know from the Rellich's lemma in elasticity (see Lemma 2.14 in \cite{BHSY}) that the homogeneous operator equation $\mathcal{A}u=0$ has only the trivial solution $u=0$. Then it follows from the Fredholm alternative that the variational formulation
(\ref{variational}) is uniquely solvable. Finally, the inf-sup condition
\be
\label{infsup}
\sup_{0\ne v\in (H^1(B_R))^2}\frac{|A(u,v)|}{\|v\|_{H^1(B_R)^2}}\ge \gamma_q\|u\|_{H^1(B_R)^2}\quad\mbox{for all}\quad u\in H^1(B_R)^2,
\en
with some constant $\gamma_q>0$ generated from the general theory in Babu\v{s}ka and Aziz\cite{BA72} implies the estimate (\ref{solestimate}). \textcolor{rot1}{In fact, for $F\in (H^1(B_R)^2)'$, considering the operator equation $\mathcal{A}u=F$ related to the variational equation (\ref{variational0}), it follows from the Babu\v{s}ka's theory that the operator equation is well-posed if and only if the conditions
\ben
\inf_{0\ne u\in (H^1(B_R))^2}\sup_{0\ne v\in (H^1(B_R))^2}\frac{|A(u,v)|}{\|u\|_{H^1(B_R)^2}\|v\|_{H^1(B_R)^2}}=C_U>0,
\enn
and
\ben
\inf_{0\ne v\in (H^1(B_R))^2}\sup_{0\ne u\in (H^1(B_R))^2}\frac{|A(u,v)|}{\|u\|_{H^1(B_R)^2}\|v\|_{H^1(B_R)^2}}=C_E>0
\enn
hold and they are equivalent to the uniqueness and existence of solution of the operator equation, respectively. Then it follows from the variational equation (\ref{variational0}) and trace theorem that
\ben
\gamma_q\|u\|_{H^1(B_R)^2} &\le& \sup_{0\ne v\in (H^1(B_R))^2}\frac{|A(u,v)|}{\|v\|_{H^1(B_R)^2}}\\
&\le& \sup_{0\ne v\in (H^1(B_R))^2} \frac{\|g\|_{H^{-1/2}(\Gamma_R)^2} \|v\|_{H^{1/2}(\Gamma_R)^2}+\|f\|_{L^2(B_R)^2}\|v\|_{L^2(B_R)^2}} {\|v\|_{H^1(B_R)^2}}\\
&\le& \sup_{0\ne v\in (H^1(B_R))^2} \frac{\|g\|_{H^{-1/2}(\Gamma_R)^2} \|v\|_{H^1(B_R)^2}+\|f\|_{L^2(B_R)^2}\|v\|_{H^1(B_R)^2}} {\|v\|_{H^1(B_R)^2}}\\
&=& \|g\|_{H^{-1/2}(\Gamma_R)^2}+\|f\|_{L^2(B_R)^2},
\enn
which completes the proof.}
\end{proof}

\begin{definition}
\label{def:stronglyelliptic}
\textcolor{rot1}{A bounded sesquilinear form $a(\cdot,\cdot)$ {on} some Hilbert space $X$ is called strongly elliptic if there
exists a compact form $q(\cdot,\cdot)$ such that}
\ben
\textcolor{rot1}{|\real\, a(u,u)|\geq C\,||u||^2_{X}-q(u,u) \qquad\mbox{for all}\quad u\in X.}
\enn
\end{definition}

\subsection{Near-field scattering map}
\label{sec:2.2}

For given perturbed parameters $q\in\mathcal{K}$, we define the scattering operator $S:\mathcal{K}\rightarrow H^1(B_R)^2$ by $S(q)=u$, where $u\in H^1(B_R)^2$ is the unique weak solution of (\ref{total})-(\ref{TBC}). It is easily seen that the map $S$ is nonlinear with respect to $q$. A direct application of Theorem \ref{wellposedness} gives the following result.

\begin{lemma}
\label{lemma.F1}
The map $S$ is bounded with the estimate
\ben
\|S(q)\|_{H^1(B_R)^2}\le \gamma_q^{-1}\|g\|_{(H^{-1/2}(\G_R))^2}.
\enn
where $c>0$ is a constant.
\end{lemma}

\begin{lemma}
\label{lemma.F2}
Given the perturbed parameters $q_l\in\mathcal{K}$, $l=1,2$, we have the estimate
\ben
\|S(q_1)- S(q_2)\|_{H^1(B_R)^2}&\le& c\gamma_{q_1}^{-1}\|q_1-q_2\|_{L^\infty(B_R)^3} \|u_2\|_{(H^1(B_R))^2},
\enn
where $u_2$ is the the unique weak solution of the problem (\ref{total})-(\ref{TBC}) with perturbed parameters $q_2$ and $c>0$ is a constant.
\end{lemma}
\begin{proof}
Let $u_1,u_2\in H^1(B_R)^2$ be the unique weak solutions of the problem (\ref{total})-(\ref{TBC}) with perturbed parameters $q_1$ and $q_2$, respectively. Set $w=u_2-u_1$.
Then we have,
\ben
a_{q_1}(w,v)-\int_{\G_R} \mathcal{B}w\cdot\overline{v}\,ds= -a_{q_2-q_1-1}(u_2,v).
\enn
Then the the inf-sup condition (\ref{infsup}) together with the boundedness of $a_{q_2-q_1-1}$ implies the desired estimate.
\end{proof}

For any $\delta q:=(\delta q_\lambda, \delta q_\mu, \delta q_\rho)^\top\in \mathcal{K}$, assume that $w\in H^1(B_R)^2$ is the unique weak solution of the following variational problem
\ben
a_{q}(w,v)-\int_{\G_R} \mathcal{B}w\cdot\overline{v}\,ds &=& -a_{\delta q-1}(u,v)\quad\mbox{for all}\quad v\in H^1(B_R)^2.
\enn
Let the map $\mathcal{T}_{q}: \mathcal{K}\rightarrow H^1(B_R)^2$ be such that
\ben
\mathcal{T}_{q}(\delta q)&=&w.
\enn

\begin{lemma}
\label{lemma.F3}
Given the perturbed parameters $q,\delta q\in\mathcal{K}$, we have the estimate
\ben
\|S(q+\delta q)-S(q) -\mathcal{T}_{q}(\delta q)\|_{H^1(B_R)^2}\le C\|\delta q\|_{L^\infty(B_R)^3}^2 \|g\|_{H^{-1/2}(\G_R)^2},
\enn
where $C>0$ is a constant.
\end{lemma}
\begin{proof}
Let $u_1,u_2\in H^1(B_R)^2$ be the unique weak solutions of the problem (\ref{total})-(\ref{TBC}) with perturbed parameters $q$ and $q+\delta q$, respectively. Let $\widetilde{u}=u_2-u_1$, it follows that
\textcolor{rot1}{\ben
a_{q}(\widetilde{u}-w,v)-\int_{\G_R} \mathcal{B}(\widetilde{u}-w)\cdot\overline{v}\,ds&=& -a_{\delta q-1}(\widetilde{u},v) \quad\mbox{for all}\quad v\in (H^1(B_R))^2.
\enn}
Then we have the estimate
\ben
\|\textcolor{rot1}{\widetilde{u}}-w\|_{H^1(B_R)^2}&\le& c\|\delta q\|_{L^\infty(B_R)^3} \|\widetilde{u}\|_{(H^1(B_R))^2}\\
&\le& c\gamma_q^{-1}\|\delta q\|_{L^\infty(B_R)^3}^2 \|u_2\|_{(H^1(B_R))^2}\\
&\le& c\gamma_q^{-1}\gamma_{q+\delta_q}^{-1}\|\delta q\|_{L^\infty(B_R)^3}^2 \|g\|_{(H^{-1/2}(\G_R))^2}\\
&\le& C\|\delta q\|_{L^\infty(B_R)^3}^2 \|g\|_{(H^{-1/2}(\G_R))^2},
\enn
where $C>0$ is a constant defined as
\ben
C:=cc_0^2,\quad c_0:=\sup_{q,\delta q\in\mathcal{K}}\{\gamma_q^{-1},\gamma_{q+\delta q}^{-1}\}<\infty.
\enn
\end{proof}

Let $\gamma: H^1(B_R)^2\rightarrow H^{1/2}(\G_R)^2$ be the trace operator to the boundary $\G_R$ and define the near-field scattering map $N$ as $N(q)=\gamma S(q)$.  By combining Lemmas \ref{lemma.F1}-\ref{lemma.F3}, we arrive at the following theorem.
\begin{theorem}
The near-field scattering map $N$ is Fr\'echet differentiable with respect to $q$ and its Fr\'echet derivative is $N'_{q}=\gamma\mathcal{T}_q$.
\end{theorem}

\section{Inverse medium problem}
\label{sec:3}

In this section, we consider the inverse medium scattering problem \textcolor{rot1}{of reconstructing} the unknown perturbed elastic parameters and develop a Landweber iterative method. Assume that the total-field data $u$ is available over a range of frequencies $\omega\in[\omega_{min},\omega_{max}]$ which can be divided into $\omega_{min}=\omega_1<\omega_2<\cdots<\omega_{N-1}<\omega_N=\omega_{max}$ and over a range of incident directions $\theta^{in}\in[\theta_{min},\theta_{max}]$ which can be divided into $\theta_{min}=\theta_1<\theta_2<\cdots<\theta_{M-1}<\theta_M=\theta_{max}$. Let $u^{i,j}:=u(x,\omega_i,\theta_j)|_{\G_R}$ be the unique solution of the direct scattering problem with $\omega=\omega_i$ and $\theta^{in}=\theta_j$. Consider the following inverse medium problem:

{\bf (IP)}: Given the elastic parameters $\lambda_0,\mu_0$ and $\rho_0$ of the background medium, reconstruct the perturbed Lam\'e parameters $q_\lambda,q_\mu$ and perturbed density $q_\rho$ from the multi-frequency measurements $u^{i,j}$, $i=1,\cdots,N$, $j=1,\cdots,M$.

The inverse problem {\bf (IP)} can be formulated as: {\it Given $\lambda_0,\mu_0$ and $\rho_0$, find $q_\lambda,q_\mu$ and $q_\rho$ such that}
\be
\label{OE}
N(q)=u^{i,j},\;q=\{q_{\lambda},q_{\mu},q_{\rho}\}\quad\mbox{for}\quad i=1,\cdots,N, \;j=1,\cdots,M.
\en
In particular, the nonlinearity and ill-posedness of the inverse problem cause mathematical challenges from both theoretical and computational points of view. The nonlinearity leads to a nonconvex optimization problem, and the ill-posedness requires certain form of regularization to get a reasonable approximation. Here, to solve the operator equation (\ref{OE}) we apply the Landweber iteration method taking the form (\cite{HNS95})
\ben
q_{k+1}=q_k+\alpha (N_{q_k}')^*(N(q)-N(q_k)),\quad k=0,1,2,\cdots,
\enn
where $\alpha$ is the step size parameter.

\subsection{Multi-frequency iterative algorithm}
\label{sec:3.1}

For the adjoint of the operator $N_q'$, we have the following result.
\begin{theorem}
\label{adjointProblem}
Let $u\in H^1(B_R)^2$ be the unique weak solution of (\ref{total})-(\ref{TBC}). Then for any $h\in H^{1/2}(\G_R)^2$,
\ben
(N_q')^*(h)= \left\{-\lambda_0(\nabla\cdot\ov{u}) (\nabla\cdot\varphi),\, -2\mu_0\mathcal{E}(\ov{u}):\mathcal{E}(\varphi),\, \rho_0\textcolor{rot1}{\omega^2}\ov{u}\cdot\varphi\right\},
\enn
where $\varphi\in H^1(B_R)^2$ is the unique weak solution of the following boundary value problem
\be
\label{ad1}
\nabla\cdot\sigma_{q_{\lambda},q_{\mu}}(\ov{\varphi})+ \rho_0(1+q_{\rho})\omega^2\ov{\varphi} &=& 0\quad\mbox{in}\quad B_R,\\
\label{ad2}
T_{\lambda_0,\mu_0}\ov{\varphi}-\mathcal{B}\ov{\varphi} &=& \ov{h}\quad\mbox{on}\quad\G_R.
\en
\end{theorem}
\begin{proof}
For any $\delta q\in \mathcal{K}$, let $w\in (H^1(B_R))^2$ be the unique weak solution of the following variational problem
\ben
a_{q}(w,v)-\int_{\G_R} \mathcal{B}w\cdot\overline{v}\,ds &=& -a_{\delta q-1}(u,v)\quad\mbox{for all}\quad v\in (H^1(B_R))^2.
\enn
Replacing $v$ by $\varphi$, we obtain
\ben
&\quad& A_{\delta q_\lambda-1}(u,\varphi)+B_{\delta q_\mu-1}(u,\varphi)- \int_{B_R} \rho_0\delta q_{\rho}\omega^2 u\cdot\ov{\varphi}dx\\
&=& -A_{q_\lambda}(w,\varphi)- B_{q_\mu}(w,\varphi)+\int_{B_R} \rho_0(1+q_{\rho})\omega^2 w\cdot\ov{\varphi}dx+ \langle\mathcal{B}w,\varphi\rangle_{-1/2,1/2}\\
&=& -A_{q_\lambda}(w,\varphi)- B_{q_\mu}(w,\varphi)+\int_{B_R} \rho_0(1+q_{\rho})\omega^2 w\cdot\ov{\varphi}dx+ \langle w,\mathcal{B}^*\varphi\rangle_{1/2,-1/2}\\
&=& -A_{q_\lambda}(w,\varphi)- B_{q_\mu}(w,\varphi)+\int_{B_R} \rho_0(1+q_{\rho})\omega^2 w\cdot\ov{\varphi}dx+ \langle w,\ov{\mathcal{B}\ov{\varphi}}\rangle_{1/2,-1/2}\\
&=& -\langle w,h\rangle_{1/2,-1/2}\\
&=& -\langle N'_q(\delta q),h)\rangle_{1/2,-1/2}\\
&=& -\int_{B_R}\delta q\cdot\ov{(N'_q)^*h}dx
\enn
in which we have used the relation (\ref{adjointDtN}) and $\langle\cdot,\cdot\rangle_{s,-s}$ denotes the $L^2$ duality pairing between $H^{s}(\G)^2$ and $H^{-s}(\G)^2$. Since it holds for any $\delta q\in \mathcal{K}$, we complete the proof.
\end{proof}

Denote
\ben
q_{i,j,l}=\{q_{\lambda}^{i,j,l}, q_{\mu}^{i,j,l}, q_{\rho}^{i,j,l}\},
\enn
\textcolor{rot1}{where the index $i,j,l$ are related to the frequency $\omega_i$, the incident direction $\theta_j$ and the current inner Landweber iteration number.} Given initial guesses \textcolor{rot1}{$q_{1,1,0}=0$}, we now describe a procedure that determines a better approximation \textcolor{rot1}{$q_{i,j}:=q_{i,j,L}$} at the frequency $\omega=\omega_i$ with incident direction $\theta^{in}=\theta_j$ for $i=1,\cdots,N$, $j=1,\cdots,M$ in an increasing manner. For each $i,j$, we apply $L$ steps of Landweber iterations, i.e., $l=1,\cdots,L$. For fixed $i,j$, suppose now that an approximation of the scatterer $q_{i,j,l-1}$ has been recovered. For the recovered scatterer $q_{i,j,l-1}$, we solve at $\omega=\omega_i$ and $\theta^{in}=\theta_j$ the direct problem
\ben
\label{rl11}
\nabla\cdot\sigma_{q_{\lambda}^{i,j,l-1},q_{\mu}^{i,j,l-1}} (\widetilde{u}^{i,j,l})+ \rho_0(1+q_{\rho}^{i,j,l-1})\omega_i^2\widetilde{u}^{i,j,l} &=&0\quad\mbox{in}\quad B_R,\\
\label{rl12}
T_{\lambda_0,\mu_0}\widetilde{u}^{i,j,l}- \mathcal{B}\widetilde{u}^{i,j,l}-g&=&0\quad\mbox{on}\quad\G_R.
\enn
Then for any $h\in H^{1/2}(\G_R)^2$,
\ben
&\quad&(N_{q_{i,j,l-1}}')^*(h)\\
&=& \left\{-\lambda_0(\nabla\cdot\ov{\widetilde{u}^{i,j,l}}) (\nabla\cdot\varphi^{i,j,l}),\, -2\mu_0\mathcal{E}(\ov{\widetilde{u}^{i,j,l}}):\mathcal{E}(\varphi^{i,j,l}),\, \rho_0\omega^2_i\ov{\widetilde{u}^{i,j,l}}\cdot\varphi^{i,j,l}\right\},
\enn
where $\varphi^{i,j,l}\in H^1(B_R)^2$ is the unique weak solution of the following boundary value problem
\ben
\label{rl21}
\nabla\cdot\sigma_{q_{\lambda}^{i,j,l-1},q_{\mu}^{i,j,l-1}} (\ov{\varphi^{i,j,l}})+ \rho_0(1+q_{\rho}^{i,j,l-1})\omega_i^2\ov{\varphi^{i,j,l}} &=& 0\quad\mbox{in}\quad B_R,\\
\label{rl22}
T_{\lambda_0,\mu_0}\ov{\varphi^{i,j,l}}- \mathcal{B}\ov{\varphi^{i,j,l}} &=& \ov{h}\quad\mbox{on}\quad\G_R.
\enn
Then the Landweber iteration leads to
\ben
q_{i,j,l}=q_{i,j,l-1}+\alpha (N_{q_{i,j,l-1}}')^*(u^{i,j}-\widetilde{u}^{i,j,l}),\quad l=1,\cdots,N.
\enn
Note that the elastic parameters are all real values. Therefore, at each step of iterations, we apply a simple regularization as
\ben
q_{i,j,l}\leftarrow \mbox{Re}\{q_{i,j,l}\}.
\enn

The multi-frequency iterative algorithm for solving the inverse medium scattering is summarized \textcolor{rot1}{in Algorithm 3.1}.

{\bf Algorithm 3.1 (Multi-frequency iterative algorithm).}
\begin{itemize}
\item Collect the near-field data over all frequencies $\om_i$, $i=1,\cdots,N$ and all incident directions $d_j$, $j=1,\cdots,M$.
\item Set initial approximations $q_{1,1,0}=0$.
\item Apply the following iteration:\\
DO $i=1,\cdots,N$\\
DO $j=1,\cdots,M$\\
DO $l=1,\cdots,L$\\
Update the elastic parameters by the formula
    \ben
    q_{i,j,l}=q_{i,j,l-1}+\alpha\mbox{Re}\left\{\begin{bmatrix}
    -\lambda_0(\nabla\cdot\ov{\widetilde{u}^{i,j,l-1}}) (\nabla\cdot\varphi^{i,j,l-1}) \\ -2\mu_0\mathcal{E}(\ov{\widetilde{u}^{i,j,l-1}}):\mathcal{E}(\varphi^{i,j,l-1}) \\ \rho_0\omega^2_i\ov{\widetilde{u}^{i,j,l-1}}\cdot\varphi^{i,j,l-1}
    \end{bmatrix}\right\}
    \enn
ENDDO\\
Set $q_{i,j+1,0}=q_{i,j,L}$\\
ENDDO\\
Set $q_{i+1,1,0}=q_{i,M,L}$\\
ENDDO
\end{itemize}

\subsection{Multifrequency iterative algorithm from phaseless data}
\label{sec:3.2}

We now consider the reconstruction from phaseless data. Define the phaseless near-field scattering map $F: \mathcal{K}\rightarrow H^{1/2}(\G_R)$ by
\ben
F(q)=\left|u|_{\G_R}\right|^2 =\ov{N(q)}\cdot N(q).
\enn

\begin{lemma}
The phaseless near-field scattering map $F$ is Fr\'echet differentiable with respect to $q$ and its Fr\'echet derivative is given by
\ben
F'_{q}(\cdot)=2\,\mbox{Re} \left\{\ov{N(q)}\cdot N'_{q}(\cdot) \right\}.
\enn
\end{lemma}

The adjoint of $F'_{q}$ is given in the following Theorem.
\begin{theorem}
Let $u\in H^1(B_R)^2$ be the unique weak solution of (\ref{total})-(\ref{TBC}). Then for any $\hbar\in H^{1/2}(\G_R)$,
\ben
(F_q')^*(\hbar)= 2\mbox{Re}\left\{-\lambda_0(\nabla\cdot\ov{u}) (\nabla\cdot\psi),\, -2\mu_0\mathcal{E}(\ov{u}):\mathcal{E}(\psi),\, \rho_0\omega^2\ov{u}\cdot\psi\right\},
\enn
where $\psi\in H^1(B_R)^2$ is the unique weak solution of the following boundary value problem
\be
\label{ad3}
\nabla\cdot\sigma_{q_{\lambda},q_{\mu}}(\ov{\varphi})+ \rho_0(1+q_{\rho})\omega^2\ov{\varphi} &=& 0\quad\mbox{in}\quad B_R,\\
\label{ad4}
T_{\lambda_0,\mu_0}\ov{\varphi}-\mathcal{B}\ov{\varphi} &=& \ov{\hbar u}\quad\mbox{on}\quad\G_R.
\en
\end{theorem}
\begin{proof}
The proof is similar as Theorem \ref{adjointProblem} and we omit it here.
\end{proof}

We now describe the Landweber iterative algorithm based on phaseless data. Use the same notations in section \ref{sec:3.1}. For any $\hbar\in H^{1/2}(\G_R)$, let $\psi^{i,j,l}\in H^1(B_R)^2$ be the unique weak solution of the following boundary value problem
\ben
\label{rl31}
\nabla\cdot\sigma_{q_{\lambda}^{i,j,l-1},q_{\mu}^{i,j,l-1}} (\ov{\psi^{i,j,l}})+ \rho_0(1+q_{\rho}^{i,j,l-1})\omega_i^2\ov{\psi^{i,j,l}} &=& 0\quad\mbox{in}\quad B_R,\\
\label{rl32}
T_{\lambda_0,\mu_0}\ov{\psi^{i,j,l}}- \mathcal{B}\ov{\psi^{i,j,l}} &=& \ov{\hbar \textcolor{rot1}{\widetilde{u}}^{i,j,l}}\quad\mbox{on}\quad\G_R.
\enn
Then the Landweber iteration leads to
\ben
q_{i,j,l}=q_{i,j,l-1}+\alpha (F_{q_{i,j,l-1}}')^*(|u^{i,j}|^2-|\widetilde{u}^{i,j,l}|_{\G_R}|^2),\quad l=1,\cdots,N.
\enn
We conclude the algorithm for solving the inverse medium scattering from phaseless data is summarized \textcolor{rot1}{in Algorithm 3.2}.

{\bf Algorithm 3.2 (Multi-frequency iterative algorithm for phaseless data).}
\begin{itemize}
\item Collect the near-field data over all frequencies $\om_i$, $i=1,\cdots,N$ and all incident directions $d_j$, $j=1,\cdots,M$.
\item Set initial approximations $q_{1,1,0}=0$.
\item Apply the following iteration:\\
DO $i=1,\cdots,N$\\
DO $j=1,\cdots,M$\\
DO $l=1,\cdots,L$\\
Update the elastic parameters by the formula
    \ben
    q_{i,j,l}=q_{i,j,l-1}+2\alpha\mbox{Re}\left\{\begin{bmatrix}
    -\lambda_0(\nabla\cdot\ov{\widetilde{u}^{i,j,l-1}}) (\nabla\cdot\psi^{i,j,l-1}) \\ -2\mu_0\mathcal{E}(\ov{\widetilde{u}^{i,j,l-1}}):\mathcal{E}(\psi^{i,j,l-1}) \\ \rho_0\omega^2_i\ov{\widetilde{u}^{i,j,l-1}}\cdot\psi^{i,j,l-1}
    \end{bmatrix}\right\}
    \enn
ENDDO\\
Set $q_{i,j+1,0}=q_{i,j,L}$\\
ENDDO\\
Set $q_{i+1,1,0}=q_{i,M,L}$\\
ENDDO
\end{itemize}

\subsection{\textcolor{rot1}{Convergence}}
\label{sec:3.3}

\textcolor{rot1}{In this section, we briefly discuss the convergence of the proposed algorithms based on the classical analysis of Landweber iteration method for nonlinear ill-posed problems. For more detailed analysis of the Landweber iteration and its modification form, we refer to \cite{HNS95,KNS} and the references therein. It should be pointed out that although we apply the multi-frequency strategy in this paper, it is extremely difficult to investigate the dependence of the convergence of the algorithms on the number and interval of the selected frequencies. For the corresponding analysis of the inverse medium scattering problems in acoustics that take a similar form as the special case discussed in section \ref{sec:4}, we refer to \cite{BT10}.}

\textcolor{rot1}{We consider the following operator equation
\be
\label{InverseOE}
F(q,\omega)=y(\omega),\quad F:\mathcal{D}(F)\times [\omega_{min},\omega_{max}]\rightarrow Y,
\en
where $\mathcal{D}(F)\subset X$ and $X,Y$ are all Hilbert spaces with inner products $(\cdot,\cdot)_X,(\cdot,\cdot)_Y$ and norms $\|\cdot\|_X,\|\cdot\|_Y$, respectively. In particular, $X=L^2(B_R)^3,Y=H^{1/2}(\Gamma_R)^2$ in this paper. As mentioned above, the operator equation (\ref{InverseOE}) is strictly nonlinear and ill-posed. For simplicity, denote $F(q)=F(q,\omega)$. Then the nonlinear Landweber iteration takes the form
\be
\label{exact}
q_{k+1}=q_k+\alpha (F_{q_k}')^*(y-F(q_k)),\quad k=0,1,2,\cdots,
\en
for exact data $y$ and the form
\be
\label{inexact}
q_{k+1}^\delta=q_k^\delta+\alpha (F_{q_k^\delta}')^*(y^\delta-F(q_k^\delta)),\quad k=0,1,2,\cdots,
\en
for inexact data $y^\delta$ satisfying $\sup_{\omega\in[\omega_{min},\omega_{max}]} \{\|y^\delta-y\|_Y\}\le\delta$. Here, assume that we take the frequency as the innermost loop.}

\textcolor{rot1}{Let $\mathcal{B}_r(q_0)$ denote a closed ball of radius $r$ around $q_0$. Assume that $\alpha>0$ is sufficiently small such that $\alpha\|F_q'\|<1$ in $\mathcal{B}_{2r}(q_0)$. For simplicity, we consider $\alpha=1$ and $\|F_q'\|<1$. Resulting from the regularity theory, we can obtain that $u\in H_{loc}^2(\R^2\backslash \ov{B_{R-\epsilon}})^2$, for any small $\epsilon>0$ such that $\mbox{sup}\{q\}\subset B_{R-\epsilon}$, which implies that $u\in H^{3/2}(\Gamma_R)^2\hookrightarrow H^{1/2}(\Gamma_R)^2$. Thus, the operator $F$ here is compact. According to the analysis in Lemma \ref{lemma.F3} and inverse trace theorem, let $r>0$ and $q,\widetilde{q}\in \mathcal{B}_{2r}(q_0)\subset \mathcal{D}(F)$ satisfying $\|q-\widetilde{q}\|_{L^\infty(B_R)^3}\le\epsilon$ with sufficiently small $\epsilon>0$, it holds that
\ben
\|F(q)-F(\widetilde{q})-F_q'(q-\widetilde{q})\|_Y\le \eta\|F(q)-F(\widetilde{q})\|_Y, \quad 0<\eta<c_0<\frac{1}{2}.
\enn
}

\textcolor{rot1}{The convergence of the Landweber iteration is given in the following theorem, see Corollary 2.3 and Theorem 2.4 in \cite{KNS}.
\begin{theorem}
Assume that $F(q)=y$ is solvable in $\mathcal{B}_r(q_0)$. Then the nonlinear Landweber iteration (\ref{exact}) converges to a solution of $F(q)=y$. Furthermore, if $\mathcal{N}(F_{q^\dagger}')\subset\mathcal{N}(F_q')$ for all $q\in\mathcal{B}_r(q^\dagger)$, then $q_k$ converges to $q^\dagger$ as $k\rightarrow+\infty$.
\end{theorem}}

\textcolor{rot1}{In case of inexact data, the iteration procedure should be combined with a stopping rule in order to act as a regularization method. For example, for each frequency, one can employ the {\it discrepancy principle}, i.e., the iteration is stopped after $k'$ steps with
\ben
\|y^\delta-F(q_{k'}^\delta)\|_Y \le\tau\delta< \|y^\delta-F(q_k^\delta)\|_Y,
\enn
where $\tau$ is an appropriately chosen positive number. Before stating the convergence of the iteration (\ref{inexact}) for inexact data case, we need to discuss the selection of $\tau$. Note that the parameter $\eta$ must be dependent on $\omega$.
\begin{proposition}
Assume $F(q)=y$ has a solution $q_*\in\mathcal{B}_r(q_0)$ and $q_k^\delta\in \mathcal{B}_r(q_*)$.
\begin{itemize}
\item[(i).] A sufficient condition for $q_{k+1}^\delta$ generated from (\ref{inexact}) to be a better approximation of $q_*$ than $q_k^\delta$ is that
\ben
\|y^\delta-F(q_k^\delta)\|_Y >2\frac{1+\eta_0}{1-2\eta_0}\delta,\quad \eta_0=\sup_{\omega\in[\omega_{min},\omega_{max}]}\{\eta\}
\enn
This further leads to the selection of $\tau$ as
\ben
\tau>2\frac{1+\eta_0}{1-2\eta_0}.
\enn
\item[(ii).] At each frequency, let $k'$ be chosen according to the stopping rule, we have
\ben
k'(\tau\delta)^2<\sum_{k=0}^{k'-1}\|y^\delta-F(q_k^\delta)\|_Y^2\le \frac{\tau}{(1-2\eta_0)\tau-2(1+\eta_0)}\|q_0-q_*\|_X^2.
\enn
\end{itemize}
\end{proposition}
\begin{proof}
The proof is similar to that of the Proposition 2.2 and Corollary 2.3 in \cite{KNS}. In fact, from (2.13) in \cite{KNS}, it holds that
\ben
&\quad& \|q_{k+1}^\delta-q_*\|_X-\|q_k^\delta-q_*\|_X \\
&\le& \|y^\delta-F(q_k^\delta)\|_Y \left(2\delta+2\eta\|y-F(q_k^\delta)\|_Y-\|y^\delta-F(q_k^\delta)\|_Y\right)\\
&\le& \|y^\delta-F(q_k^\delta)\|_Y \left(2\delta+2\eta_0\|y-F(q_k^\delta)\|_Y-\|y^\delta-F(q_k^\delta)\|_Y\right)\\
&\le& \|y^\delta-F(q_k^\delta)\|_Y \left[2(1+\eta_0)\delta -(1-2\eta)\|y^\delta-F(q_k^\delta)\|_Y \right]
\enn
which completes the proof of (i) under the sufficient condition. Furthermore, the special choice of $\tau$ implies that
\ben
&\quad& \|q_{k+1}^\delta-q_*\|_X-\|q_k^\delta-q_*\|_X \\
&\le& \left[2\tau^{-1}(1+\eta_0)+2\eta_0-1\right]\|y^\delta-F(q_k^\delta)\|_Y^2.
\enn
Adding up these inequalities for $k$ from 0 through $k'-1$ leads to
\ben
&\quad& \left[1-2\eta_0-2\tau^{-1}(1+\eta_0)\right] \sum_{k=0}^{k'-1}\|y^\delta-F(q_k^\delta)\|_Y^2\\
&\le& \|q_0-q_*\|_X^2-\|q_{k'}^\delta-q_*\|_X^2.
\enn
Then (ii) can be proved by combining this inequality and the stoping rule.
\end{proof}}

\textcolor{rot1}{Now we can have the following convergence result for inexact data, see Theorem 2.6 in \cite{KNS}.
\begin{theorem}
Assume that $F(q)=y$ is solvable in $\mathcal{B}_r(q_0)$ and let $k'$ be chosen according to the stopping rule. The nonlinear Landweber iteration (\ref{inexact}) converges to a solution of $F(q)=y$. If $\mathcal{N}(F_{q^\dagger}')\subset\mathcal{N}(F_q')$ for all $q\in\mathcal{B}_r(q^\dagger)$, then $q_{k'}^\delta$ converges to $q^\dagger$ as $\delta\rightarrow 0$.
\end{theorem}}

\textcolor{rot1}{\begin{remark}
According to the convergence of Landweber iteration, the choice of the initial approximation may affect the efficiency of the algorithms. In fact, in this paper we reconstruct the relative factors of the elastic parameters of the inhomogeneous medium compared with the parameters of the background medium, not the elastic parameters of the inhomogeneous medium themselves, i.e., we reconstruct
\ben
q_\lambda=\frac{\lambda}{\lambda_0}-1,\quad q_\mu=\frac{\mu}{\mu_0}-1,\quad q_\rho=\frac{\rho}{\rho_0}-1,
\enn
and they are of $\mathcal{O}(1)$. Thus, a simple choice of the initial approximation, that we applied in this paper, is $q_\lambda=q_\mu=q_\rho=0$ which means that there is no inhomogeneity and it can be seen from the numerical examples that this choice works well. As suggested in \cite{BLLT} for the inverse medium problems in acoustics, one can obtain an initial guess from the Born approximation and the scattering data with the higher frequency must be used in order to recover more Fourier modes of the true scatterer.
\end{remark}}

\section{Discussion of a special case}
\label{sec:4}

In this section, we briefly discuss the special case that $q_\lambda=q_\mu=0$. In this case, the considered inverse medium problem in elasticity is consistent with that in acoustics and electromagnetics discussed in \cite{BL04,BL052,BL051,BL07,BL09}.

Alternatively, one can consider near-field scattering map $\widetilde{N}$ as $\widetilde{N}(q_\rho)=\textcolor{rot1}{\widetilde{u}}|_{\Gamma_R}$, where $u=u^{sc}+u^{in}\in H^1(B_R)^2$ is the unique weak solution of (\ref{total})-(\ref{TBC}). Following the steps discussed in sections \ref{sec:2} and \ref{sec:3} (see also in \cite{BL04}), we can obtain the following result.

\begin{theorem}
The near-field scattering maps $N,\widetilde{N}$ are Fr\'echet differentiable with respect to $q_\rho$ and their Fr\'echet derivatives are given by
\ben
N'_{q_\rho}(\delta q_\rho)=\widetilde{N}'_{q_\rho}(\delta q_\rho)=v|_{\Gamma_R},
\enn
where $v\in H^1(B_R)^2$ is the unique weak solution of the following boundary value problem
\ben
\Delta^{*}_{\lambda_0,\mu_0}v+ \rho_0\omega^2(1+q_\rho)v &=& -\rho_0\omega^2\delta q_\rho(u^{in}+u^{sc})\quad\mbox{in}\quad B_R,\\
T_{\lambda_0,\mu_0}v &=& \mathcal{B}v \quad\mbox{on}\quad\Gamma_R.
\enn
Moreover, for any $h\in H^{1/2}(\G_R)^2$,
\ben
(N_{q_\rho}')^*(h)=(\widetilde{N}_{q_\rho}')^*(h)= \rho_0\textcolor{rot1}{\omega^2}\ov{(u^{in}+u^{sc})}\cdot\varphi,
\enn
where $\varphi\in H^1(B_R)^2$ is the unique weak solution of the following boundary value problem
\ben
\Delta^*_{\lambda_0,\mu_0}\ov{\varphi}+ \rho_0(1+q_{\rho})\omega^2\ov{\varphi} &=& 0\quad\mbox{in}\quad B_R,\\
T_{\lambda_0,\mu_0}\ov{\varphi}-\mathcal{B}\ov{\varphi} &=& \ov{h}\quad\mbox{on}\quad\G_R.
\enn
\end{theorem}

Similar as Algorithm 3.1, we can summarize the multi-frequency iterative algorithm for the reconstruction of inhomogeneous density function $q_\rho$ \textcolor{rot1}{in Algorithm 3.3}.

{\bf Algorithm 3.3 (Multi-frequency iterative algorithm for $q_\rho$)}
\begin{itemize}
\item Collect the near-field data over all frequencies $\om_i$, $i=1,\cdots,N$ and all incident directions $d_j$, $j=1,\cdots,M$.
\item Set an initial approximation $q_\rho^{1,1,0}=0$.
\item Apply the following iteration:\\
DO $i=1,\cdots,N$\\
DO $j=1,\cdots,M$\\
DO $l=1,\cdots,L$\\
Update the elastic parameters by the formula
    \ben
    q_\rho^{i,j,l}=q_\rho^{i,j,l-1}+\alpha\mbox{Re}\left\{  \rho_0\omega^2_i\ov{\widetilde{u}^{i,j,l-1}}\cdot\varphi^{i,j,l-1} \right\},\quad l=1,\cdots,L,
    \enn
ENDDO\\
Set $q_\rho^{i,j+1,0}=q_\rho^{i,j,L}$\\
ENDDO\\
Set $q_\rho^{i+1,1,0}=q_\rho^{i,M,L}$\\
ENDDO
\end{itemize}

\section{Numerical examples}
\label{sec:5}
In this section, we present several numerical examples to verify the accuracy and effectiveness of the recursive algorithm.  We always choose $\lambda_0=2$, $\mu_0=1$, $\rho_0=1$ and $R=1$. The near-field measurements are obtained by using finite element method solving the forward scattering problem. Define the relative error
\ben
e_q:=\frac{\|q-\widetilde{q}\|_{L^2(B_R)}}{\|q\|_{L^2(B_R)}},
\enn
where $q$ and $\widetilde{q}$ are the true and reconstructed \textcolor{rot1}{values} of the parameter of the scatterer. The true values of the elastic parameters $q_\lambda,q_\mu,q_\rho$ are shown in Fig.\ref{fig1}. Ten equally spaced frequencies are used in the construction, starting from the lowest frequency $\omega_{\min}=1$ and ending at the highest frequency $\omega_{\max}=10$. The number of incident directions is taken as $M=16$ and $\theta_j=2(j-1)\pi/M$ for $j=1,\cdots,M$. At each incident direction, 10 Landweber iteration steps are taken for one frequency. Corresponding to the stiffness tensor of the background elastic medium using Voigt notation, we chose the relaxation parameter $\alpha$ as a matrix
\be
\label{alpha1}
\alpha=\frac{1}{100\omega}\begin{bmatrix}
2+\lambda_0/\mu_0 & \lambda_0/\mu_0 & 0 \\
\lambda_0/\mu_0 & 2+\lambda_0/\mu_0 & 0 \\
0 & 0 & 1
\end{bmatrix}.
\en
\textcolor{rot1}{We collect the final reconstruction error for the following examples in Table 1 and 2. In addition, the number of iterations shown in the following figures are from 0 to $N\times M$.}

\begin{table}[htb]
\caption{\textcolor{rot1}{Final reconstruction errors of $q_\lambda$, $q_\mu$ and $q_\rho$ for Examples 1-4.}}
\centering
\begin{tabular}{|c|c|c|c|}
\hline
Figure & Error of $q_\lambda$ & Error of $q_\mu$ & Error of $q_\rho$ \\
\hline
\ref{fig2}  & 0.74 & 0.36 & 0.39 \\
\hline
\ref{fig3}  & 0.40 & 0.44 & 0.39 \\
\hline
\ref{fig4}(a,b,c)  & 0.43 & 0.47 & 0.42 \\
\hline
\ref{fig4}(d,e,f)  & 0.44 & 0.47 & 0.42 \\
\hline
\ref{fig5}  & 0.47 & 0.51 & 0.48 \\
\hline
\ref{fig6}  & 0.43 & 0.46 & 0.37 \\
\hline
\ref{fig7}  & 0.66 & 0.73 & 0.61 \\
\hline
\end{tabular}
\label{Table1}
\end{table}

\begin{figure}[htb]
\centering
\begin{tabular}{ccc}
\includegraphics[scale=0.12]{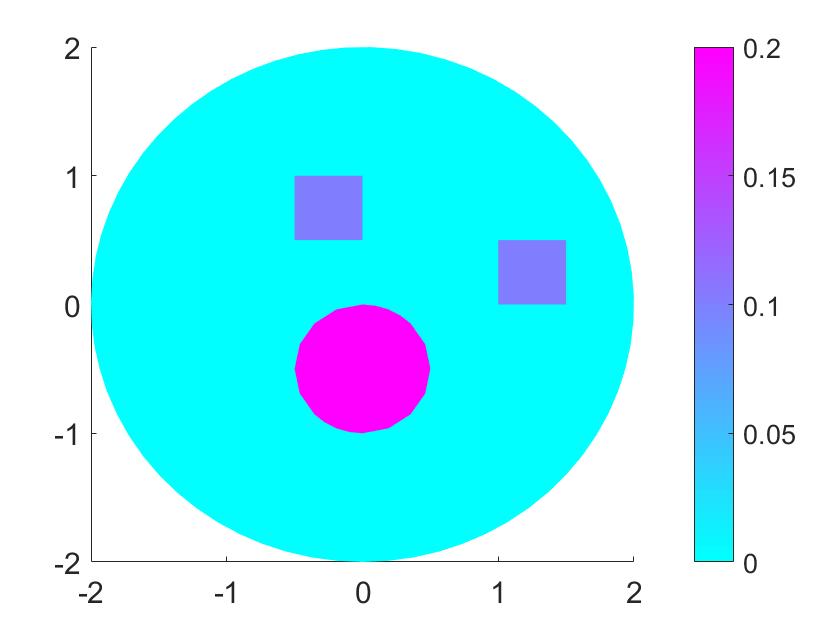} &
\includegraphics[scale=0.12]{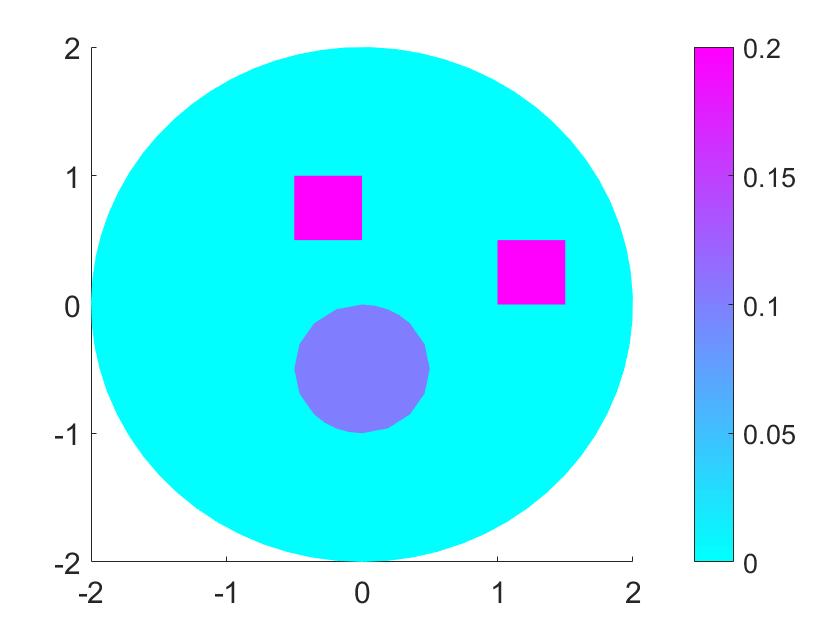} &
\includegraphics[scale=0.12]{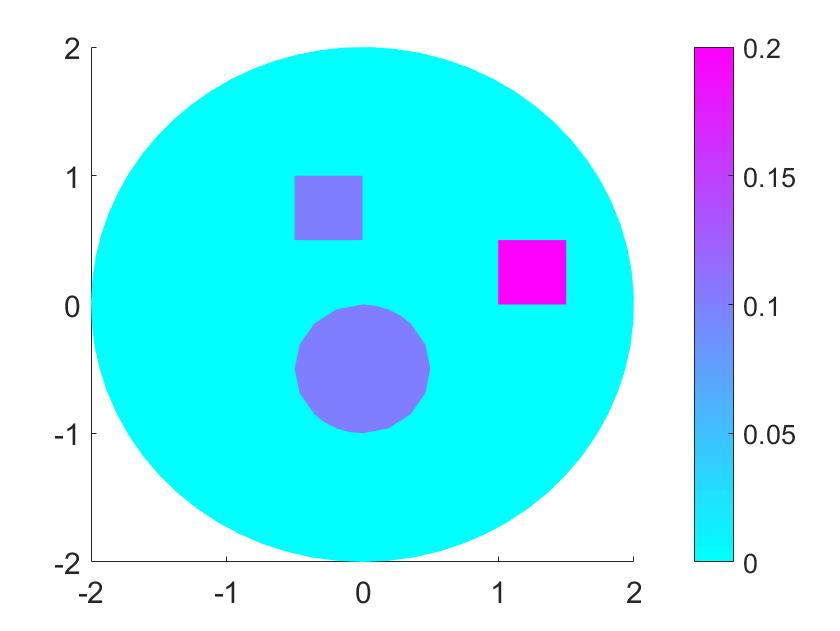} \\
(a) $q_\lambda$ & (b) $q_\mu$ & (c) $q_\rho$
\end{tabular}
\caption{The exact value of perturbed parameters.}
\label{fig1}
\end{figure}

{\bf Example 1.} We consider the reconstruction from multi-frequency measurements with multiple incident directions. The reconstructed elastic parameters are presented in Fig.\ref{fig2}(a,b,c) and Fig.\ref{fig3}(a,b,c). The relative errors shown in Fig.\ref{fig2}(d,e,f) and Fig.\ref{fig3}(d,e,f) indicate that the relative errors decrease as frequency and number of iteration increase. However, it can be seen that if we choose plane shear incident waves, the reconstruction of $q_\lambda$, which can further affect the reconstructions of $q_\mu$ and $q_\rho$, is \textcolor{rot1}{worse} than that using plane pressure incident waves. A possible explanation to this phenomenon is that, in comparison with the plane pressure waves, the plane shear incident waves only contain the information of $\mu_0$ and $\rho_0$. In the following, we consider the plane pressure incident waves only. To verify the stability of our method, the reconstructions from noised data with noise levels $\delta=3\%,5\%$ are presented in Fig.\ref{fig4}.

\begin{figure}[htb]
\centering
\begin{tabular}{ccc}
\includegraphics[scale=0.12]{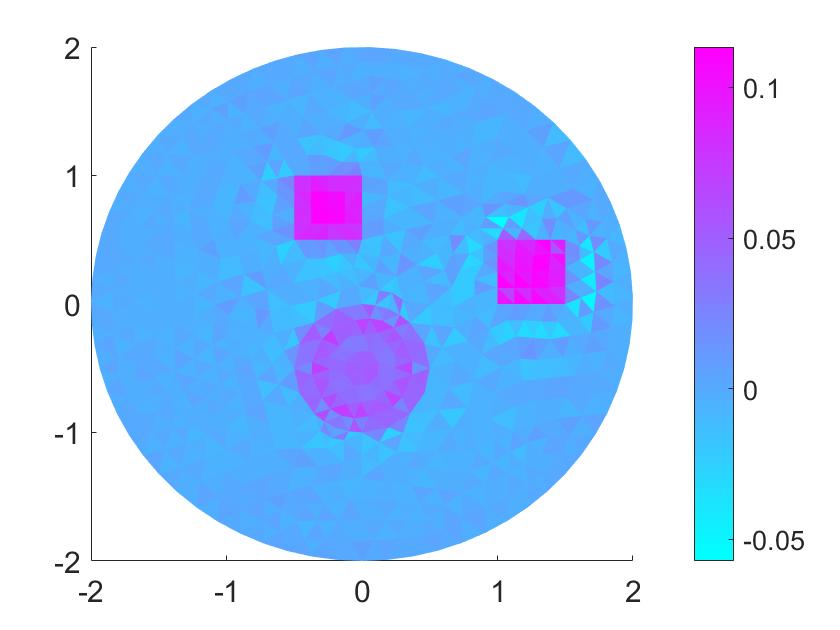} &
\includegraphics[scale=0.12]{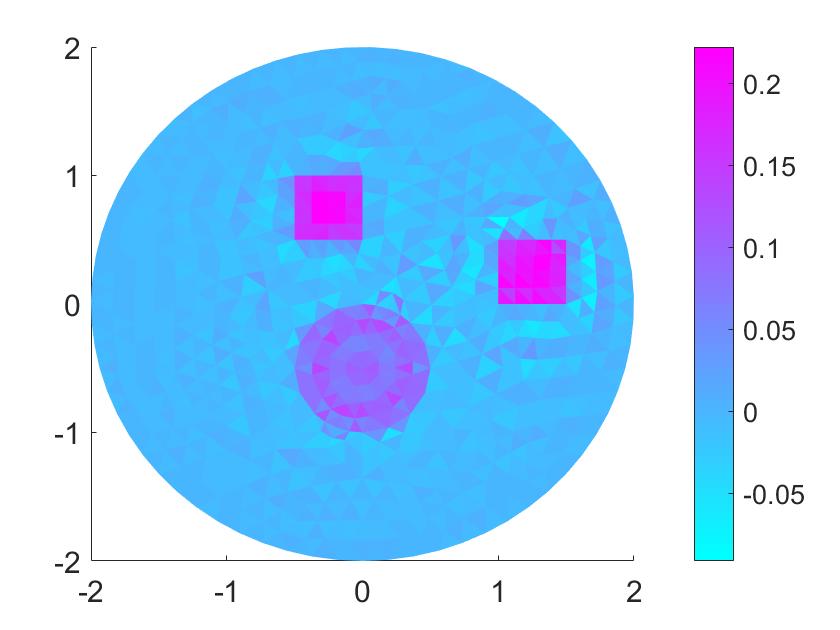} &
\includegraphics[scale=0.12]{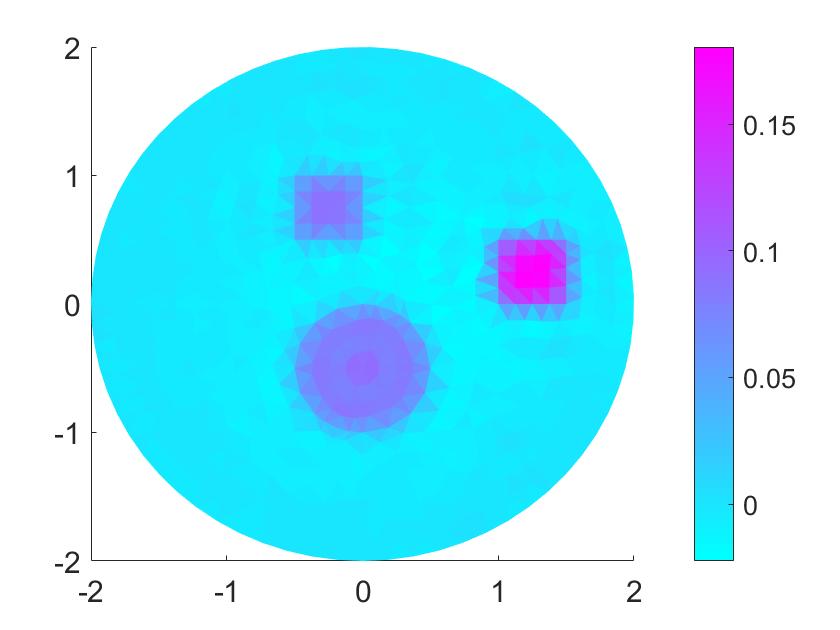} \\
(a) $q_\lambda$ & (b) $q_\mu$ & (c) $q_\rho$\\
\includegraphics[scale=0.12]{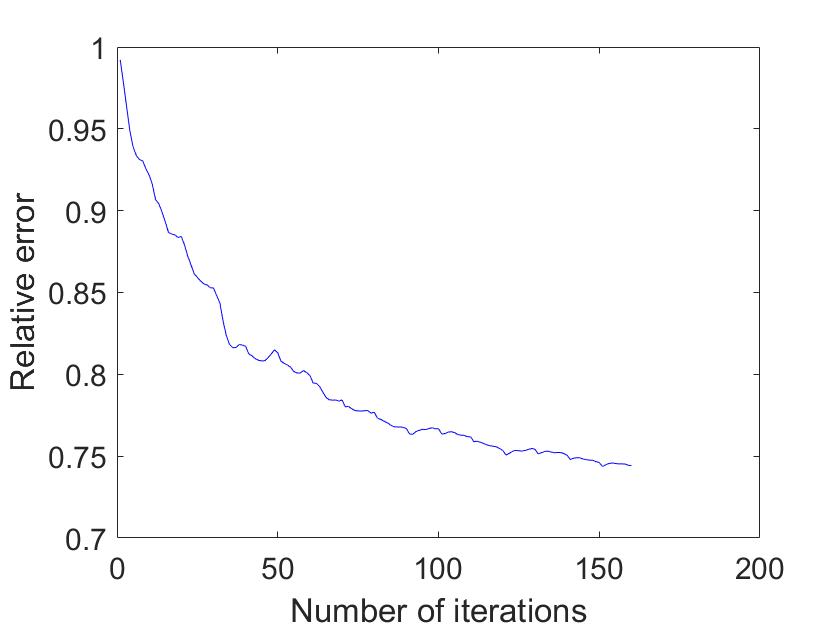} &
\includegraphics[scale=0.12]{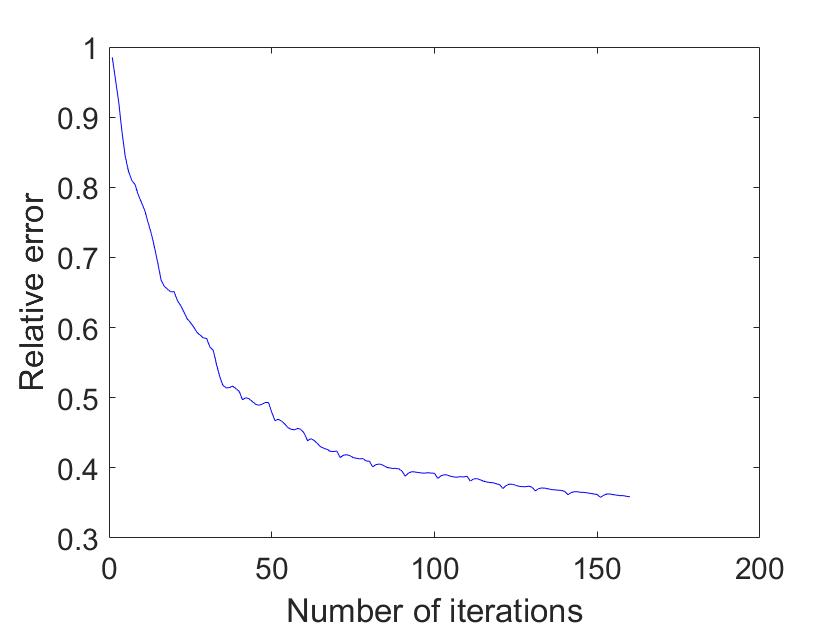} &
\includegraphics[scale=0.12]{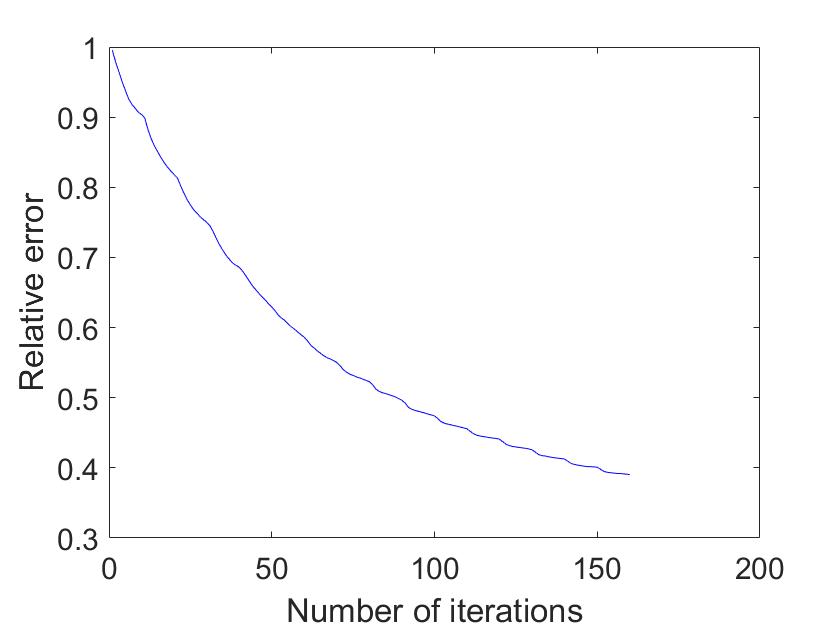} \\
(d) $e_{q_\lambda}$ & (e) $e_{q_\mu}$ & (f) $e_{q_\rho}$
\end{tabular}
\caption{Example 1: the reconstruction of perturbed parameters with plane shear incident waves.}
\label{fig2}
\end{figure}

\begin{figure}[htb]
\centering
\begin{tabular}{ccc}
\includegraphics[scale=0.12]{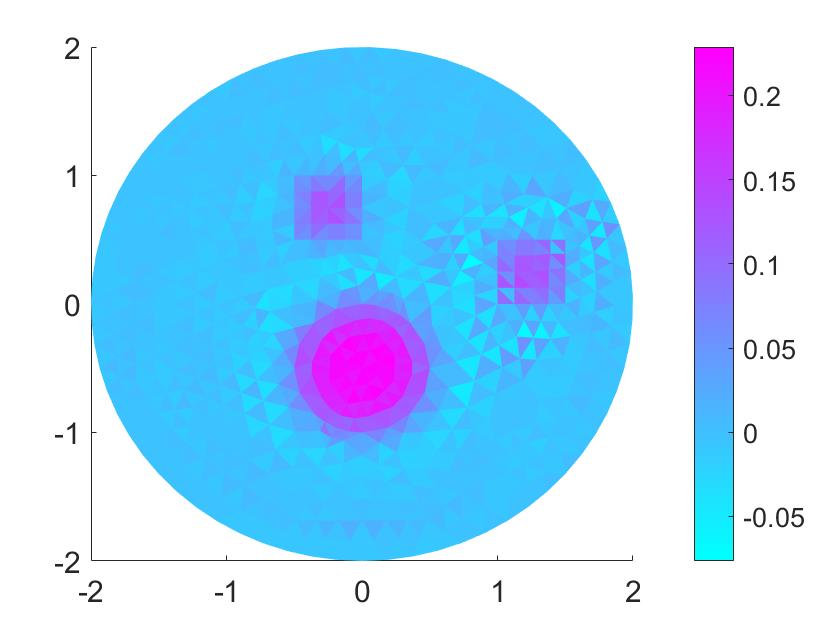} &
\includegraphics[scale=0.12]{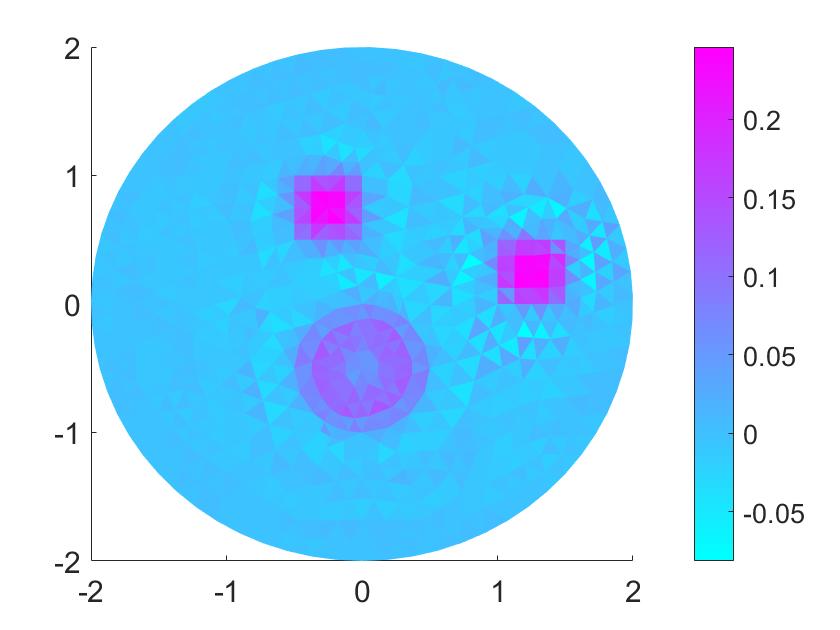} &
\includegraphics[scale=0.12]{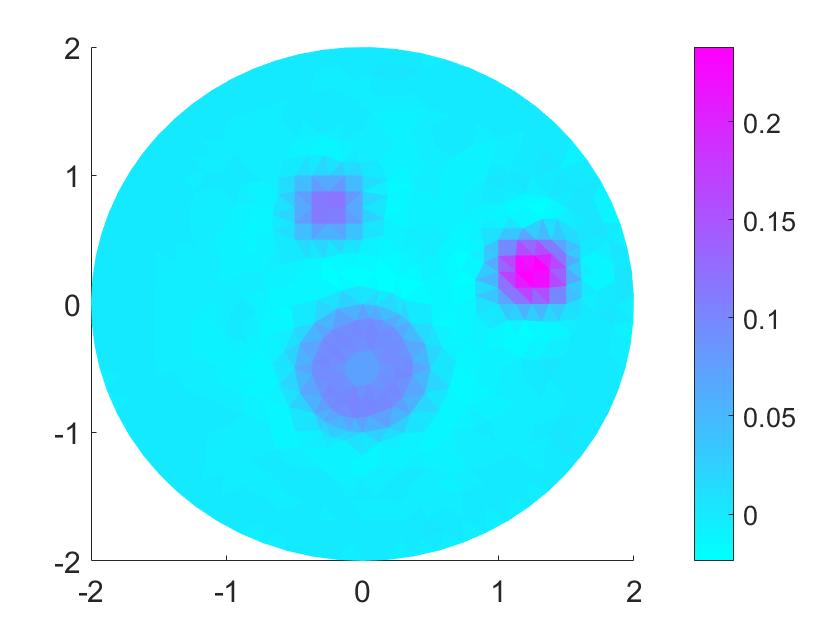} \\
(a) $q_\lambda$ & (b) $q_\mu$ & (c) $q_\rho$\\
\includegraphics[scale=0.12]{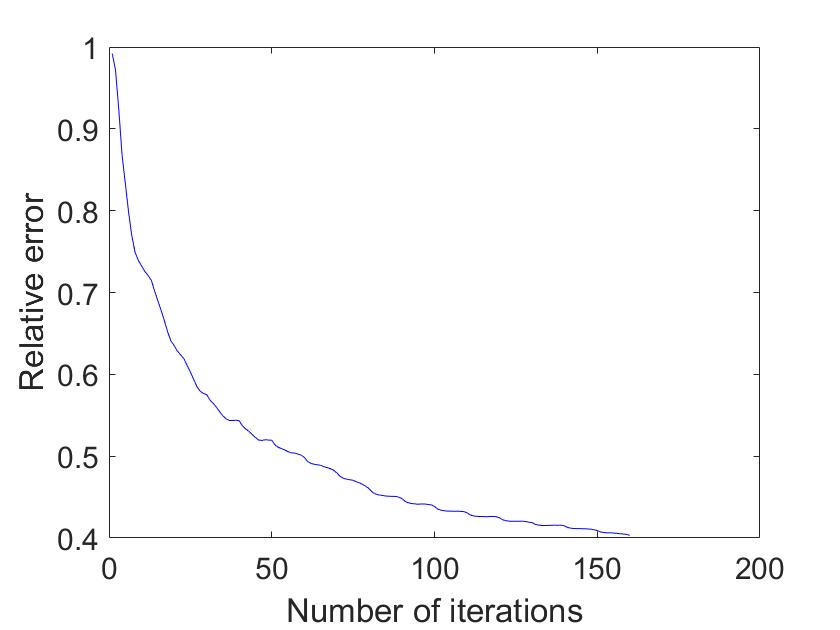} &
\includegraphics[scale=0.12]{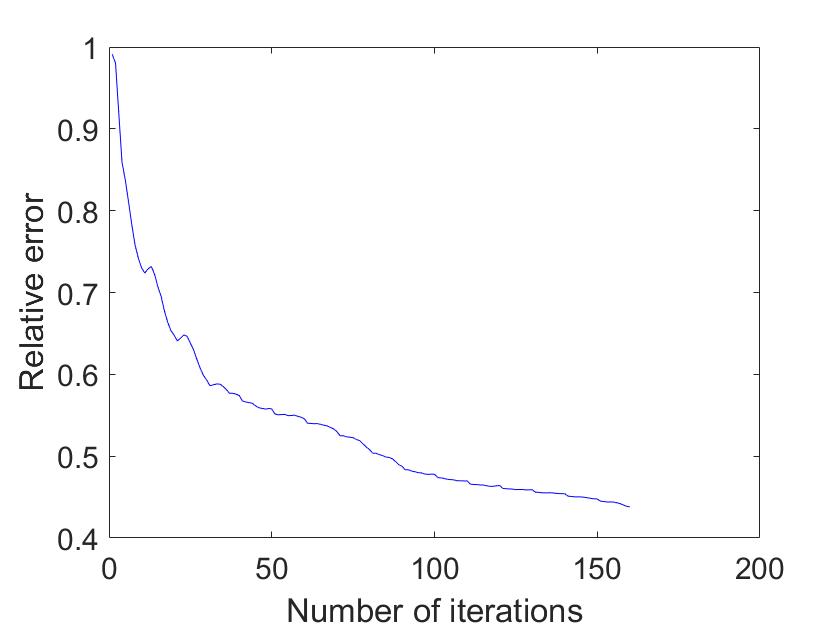} &
\includegraphics[scale=0.12]{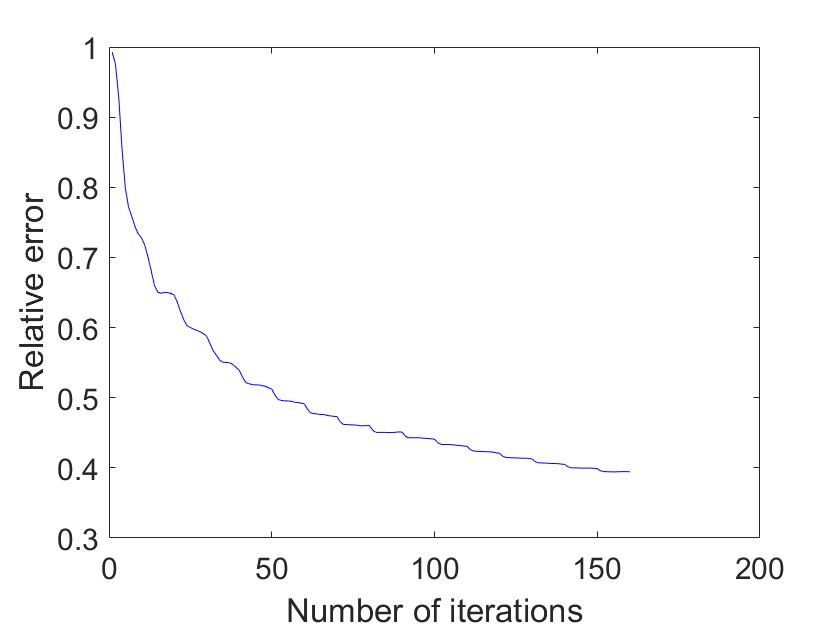} \\
(d) $e_{q_\lambda}$ & (e) $e_{q_\mu}$ & (f) $e_{q_\rho}$
\end{tabular}
\caption{Example 1: The reconstruction of perturbed parameters with plane pressure incident waves.}
\label{fig3}
\end{figure}

\begin{figure}[htb]
\centering
\begin{tabular}{ccc}
\includegraphics[scale=0.12]{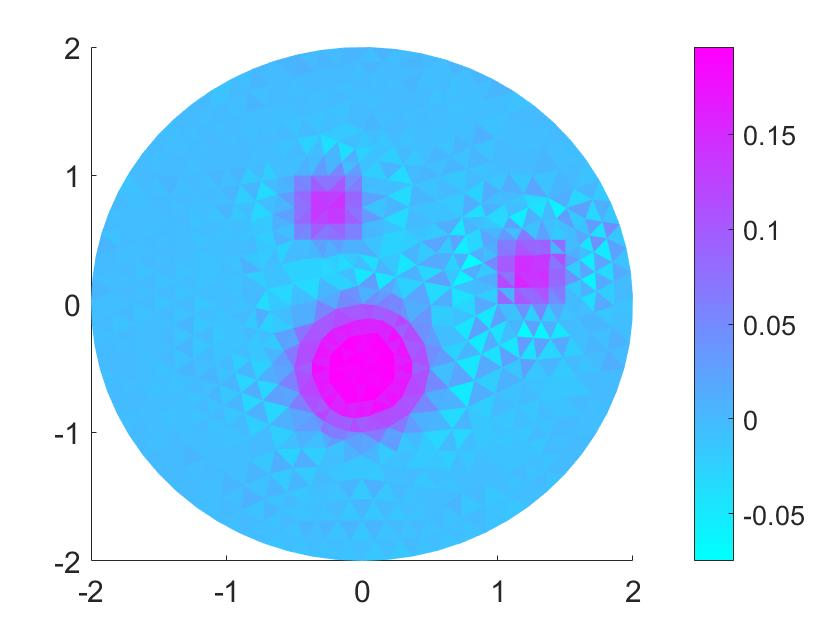} &
\includegraphics[scale=0.12]{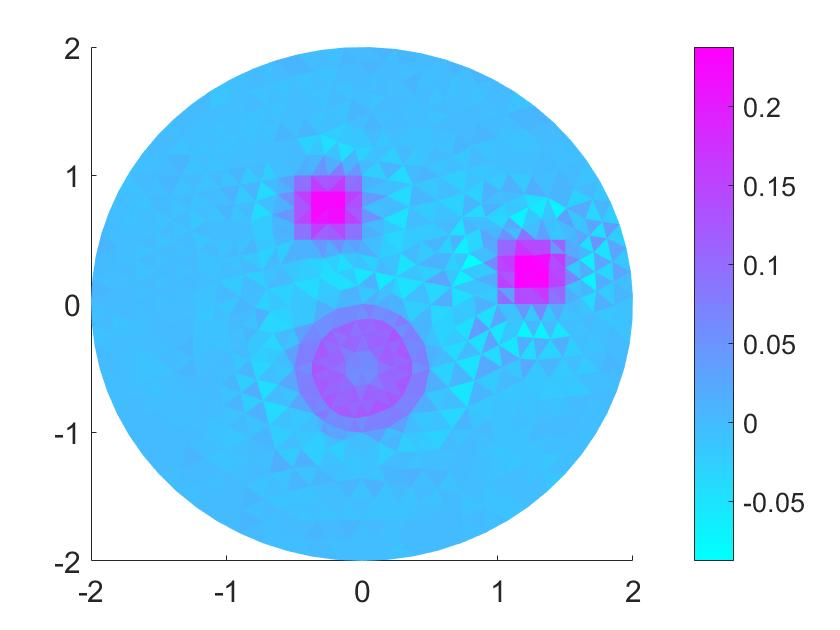} &
\includegraphics[scale=0.12]{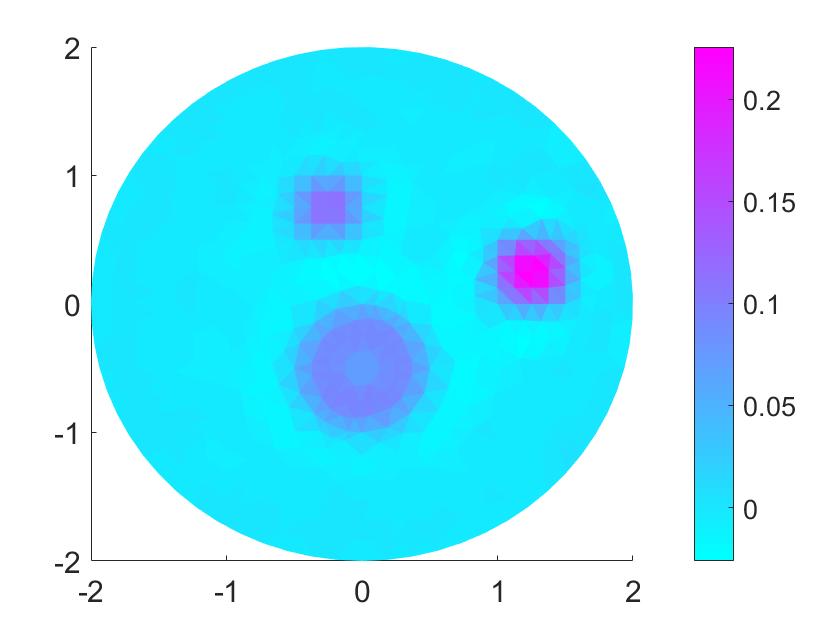} \\
(a) $q_\lambda$ & (b) $q_\mu$ & (c) $q_\rho$\\
\includegraphics[scale=0.12]{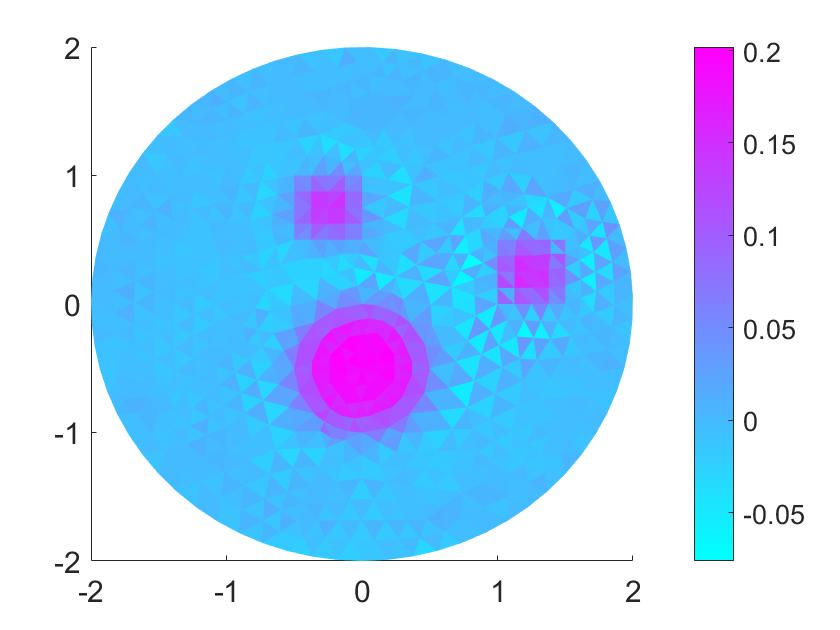} &
\includegraphics[scale=0.12]{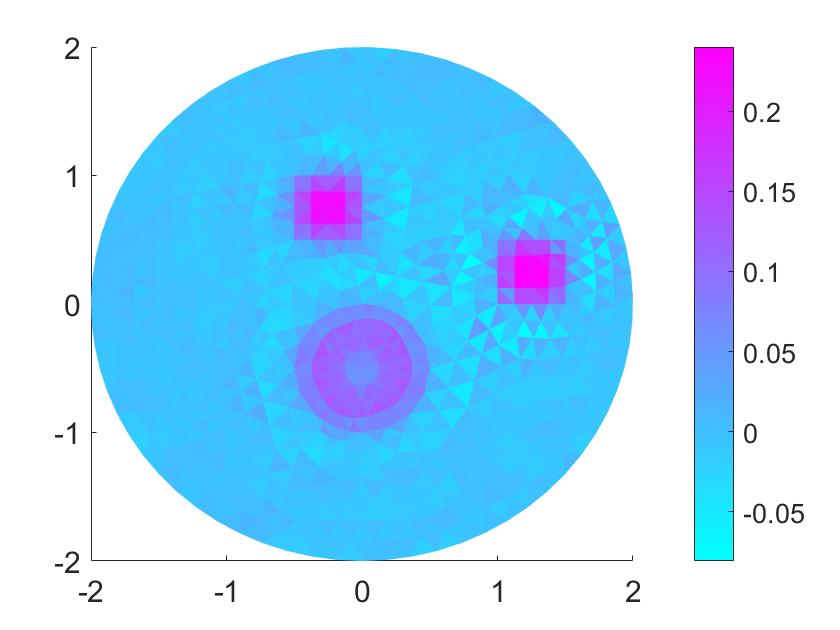} &
\includegraphics[scale=0.12]{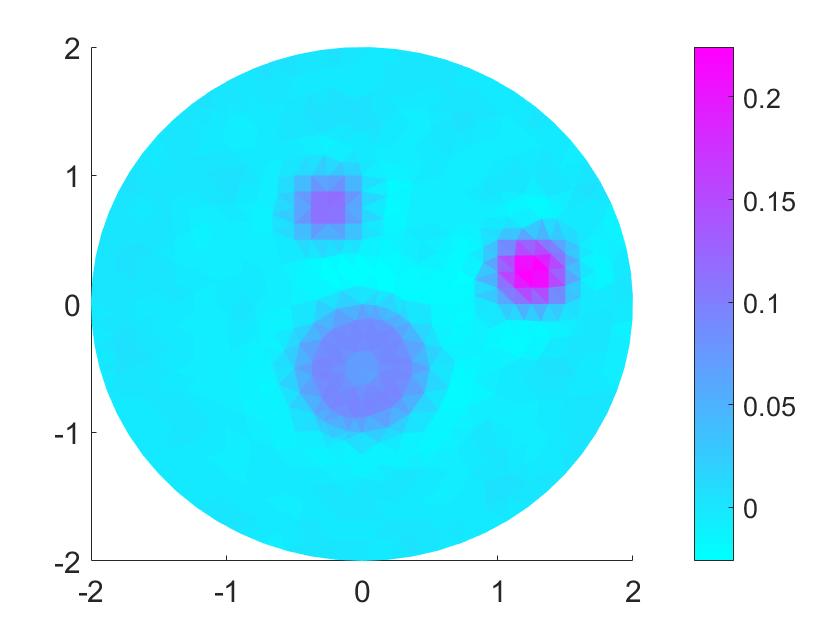} \\
(d) $q_\lambda$ & (e) $q_\mu$ & (f) $q_\rho$\\
\end{tabular}
\caption{Example 1: The reconstruction of perturbed parameters with noise level $\delta=3\%$ (a,b,c) and $\delta=5\%$ (d,e,f).}
\label{fig4}
\end{figure}

{\bf Example 2.} Note that one can chose $\alpha$ as a scale value. For simplicity, we chose $\alpha=0.01/\omega$ and the reconstruction results and the corresponding relative errors are presented in Fig.\ref{fig5}. \textcolor{rot1}{It can be seen that in this example, the iterations using the step size matrix (\ref{alpha1}) is more stable than those using scale step size. It should be pointed out that we can not prove that using the special choice of step size $\alpha$ which looks like the stiffness tensor of the background elastic medium, we can always have better reconstruction than using just a scale one. Our starting pointing is that for isotropic elastic medium, the two Lam\'e parameters have some connection since they are both determined by the Young's modulus and Poisson's ratio and have no direct relation with density. Thus, we take a similar form of stiffness tensor as the choice of step size for example which means that the modification for both $q_\lambda$, $q_\mu$ are determined by the combinations of the first two components of $(N_q')*(h)$ and the modification for density $q_\rho$ is only determined by the third component of $(N_q')*(h)$. There should be better choices of $\alpha$ than the special one we used.}

\begin{figure}[htb]
\centering
\begin{tabular}{ccc}
\includegraphics[scale=0.12]{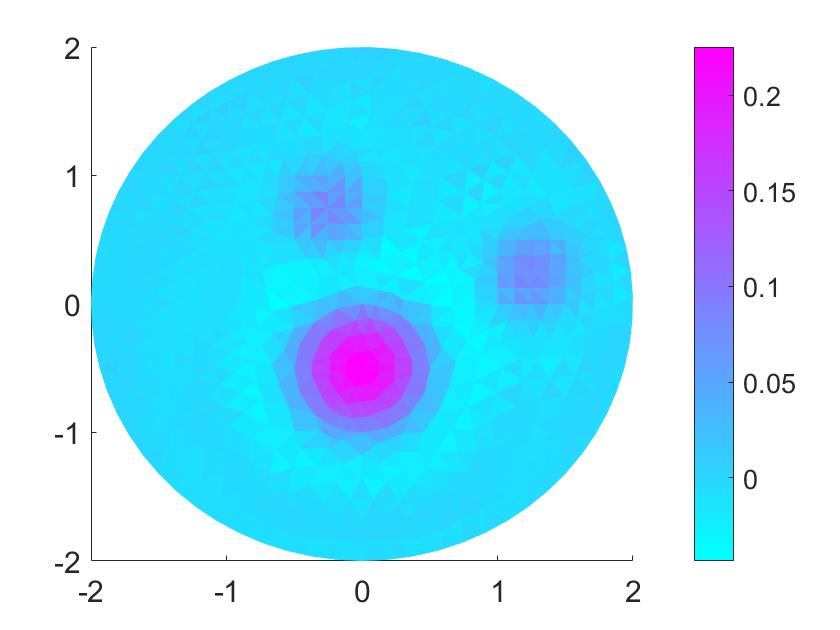} &
\includegraphics[scale=0.12]{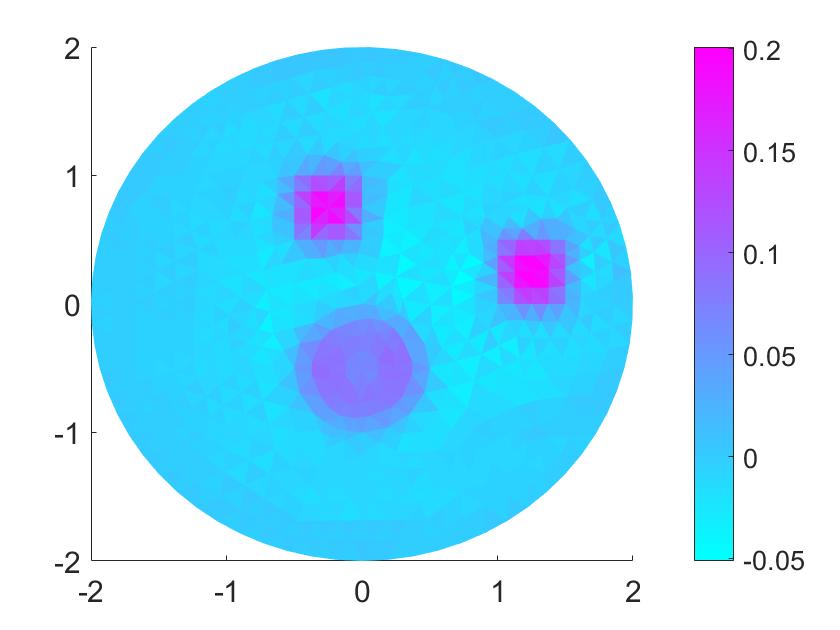} &
\includegraphics[scale=0.12]{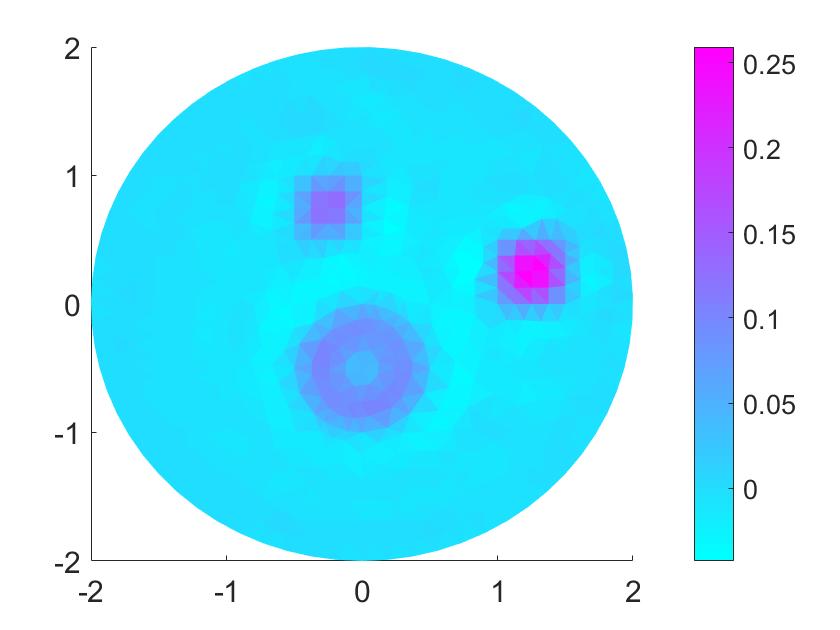} \\
(a) $q_\lambda$ & (b) $q_\mu$ & (c) $q_\rho$\\
\includegraphics[scale=0.12]{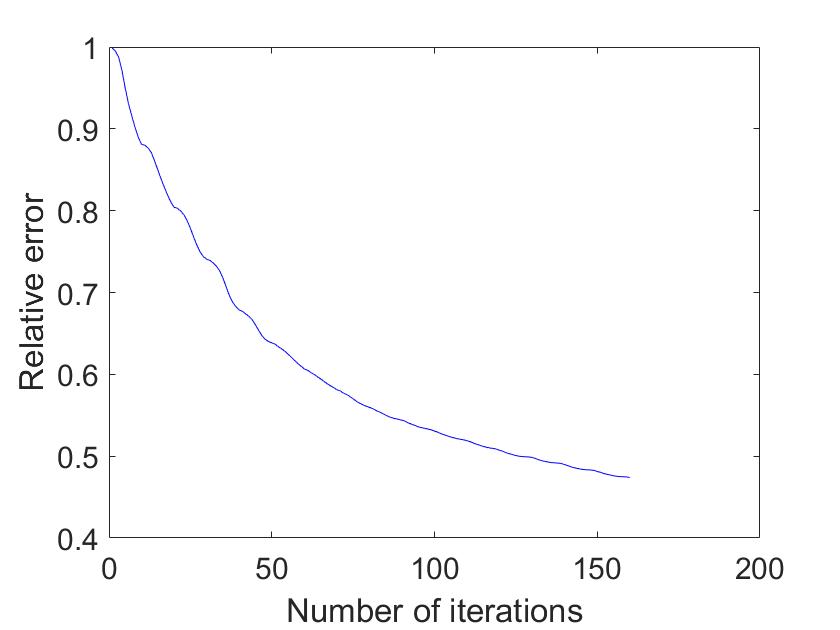} &
\includegraphics[scale=0.12]{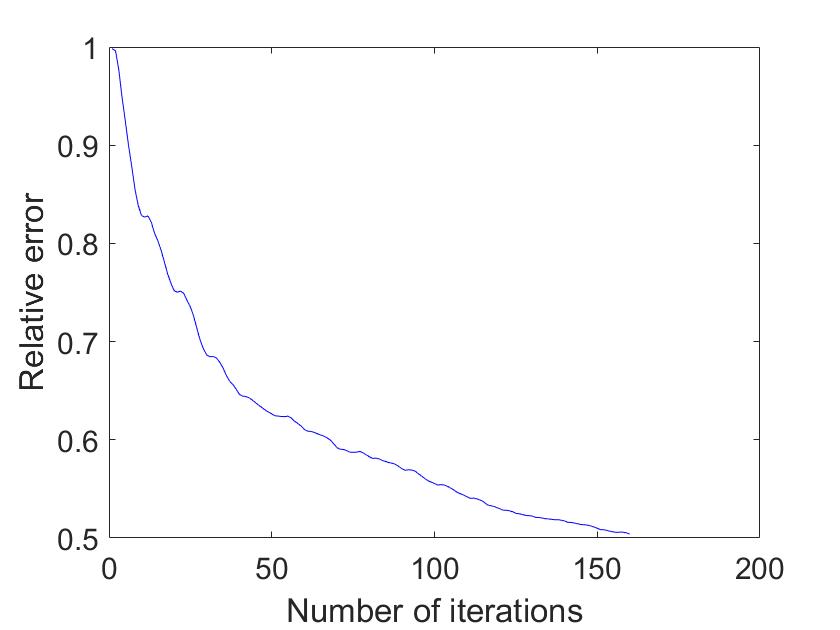} &
\includegraphics[scale=0.12]{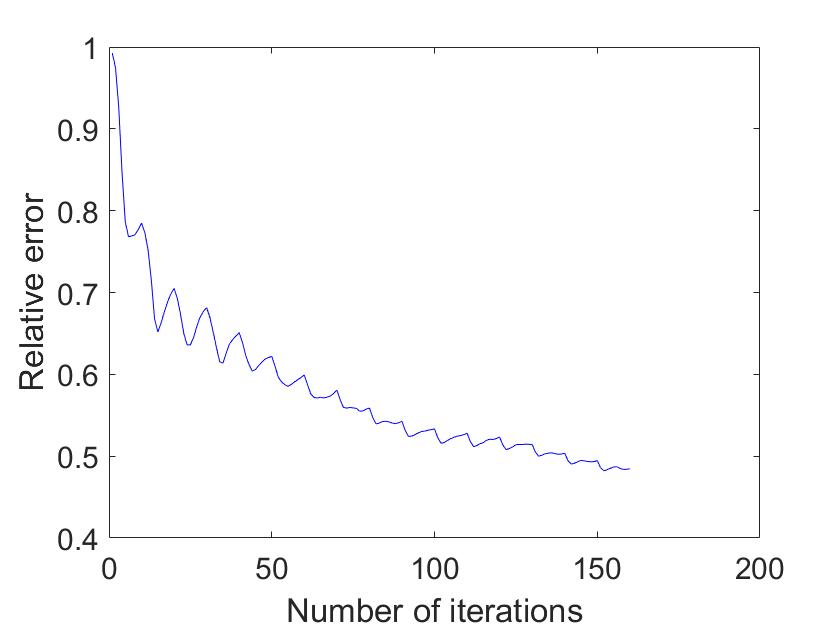} \\
(d) $e_{q_\lambda}$ & (e) $e_{q_\mu}$ & (f) $e_{q_\rho}$
\end{tabular}
\caption{Example 2: the reconstruction of perturbed parameters.}
\label{fig5}
\end{figure}

{\bf Example 3.} In this exampe, we consider the reconstruction from phaseless data. The numerical results are shown in Fig.\ref{fig6} which indicate the effectiveness of our method for the reconstruction from phaseless data.

\begin{figure}[htb]
\centering
\begin{tabular}{ccc}
\includegraphics[scale=0.12]{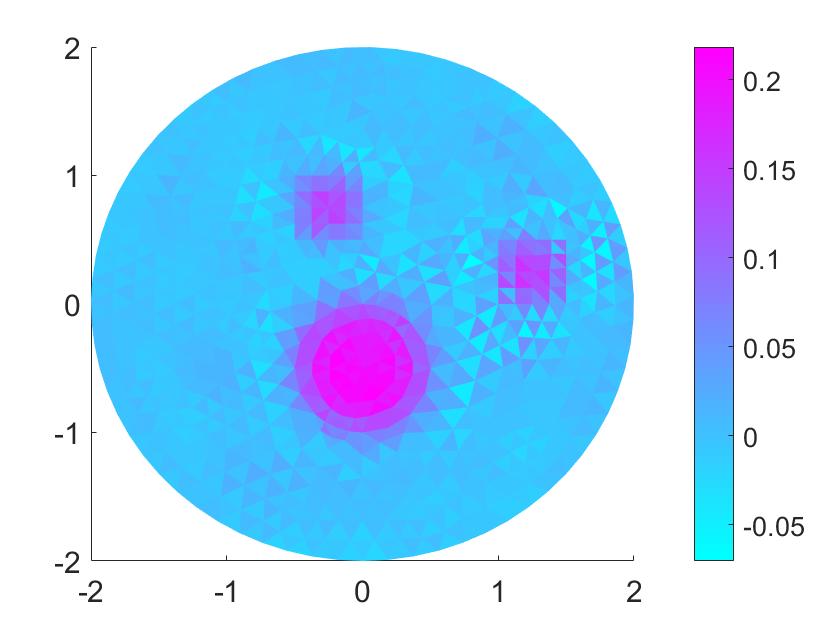} &
\includegraphics[scale=0.12]{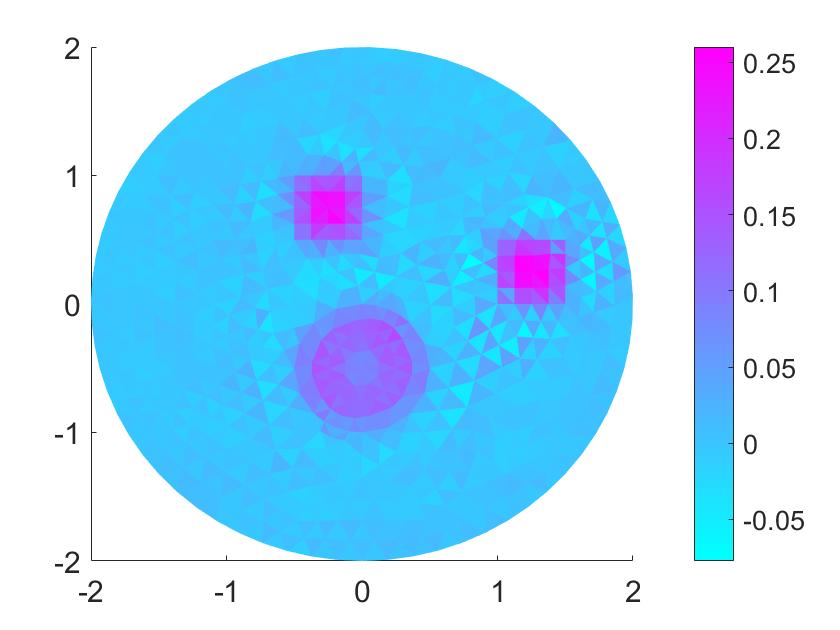} &
\includegraphics[scale=0.12]{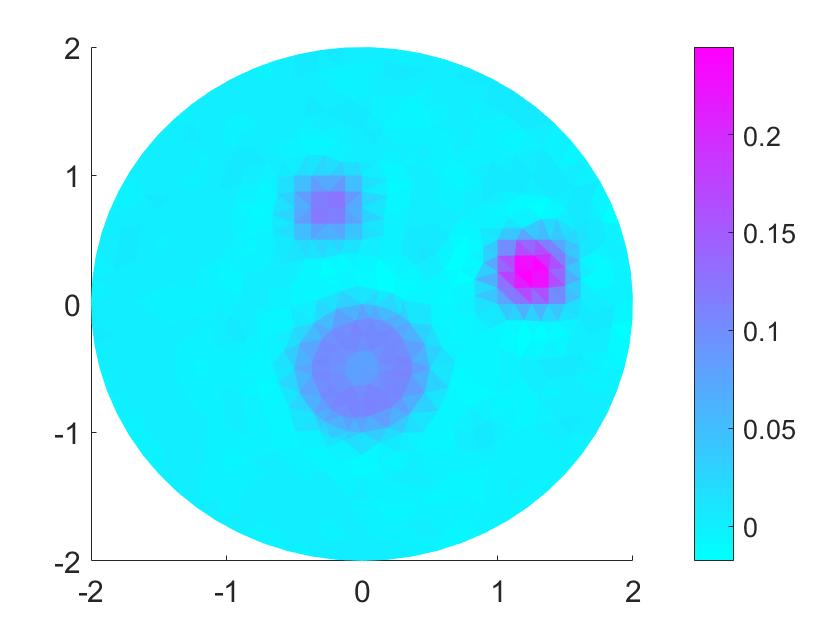} \\
(a) $q_\lambda$ & (b) $q_\mu$ & (c) $q_\rho$ \\
\includegraphics[scale=0.12]{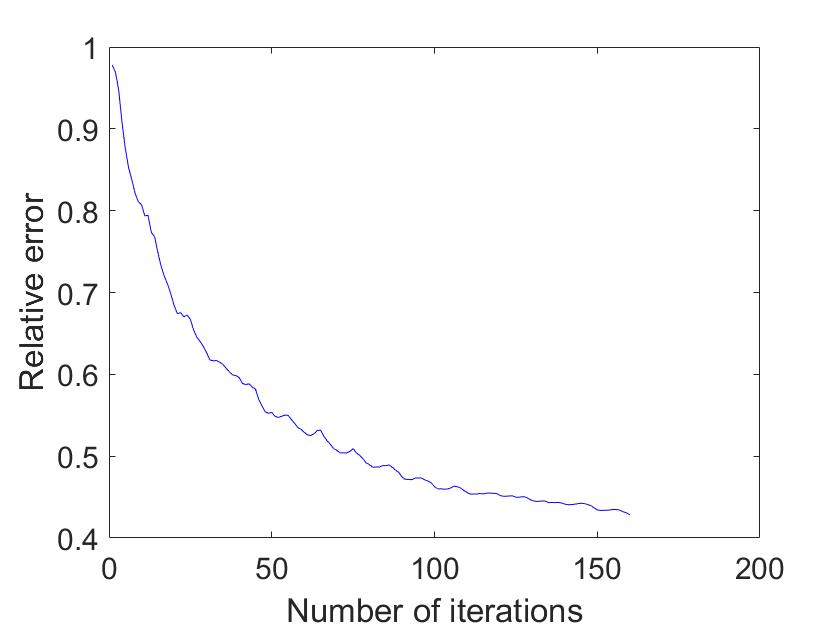} &
\includegraphics[scale=0.12]{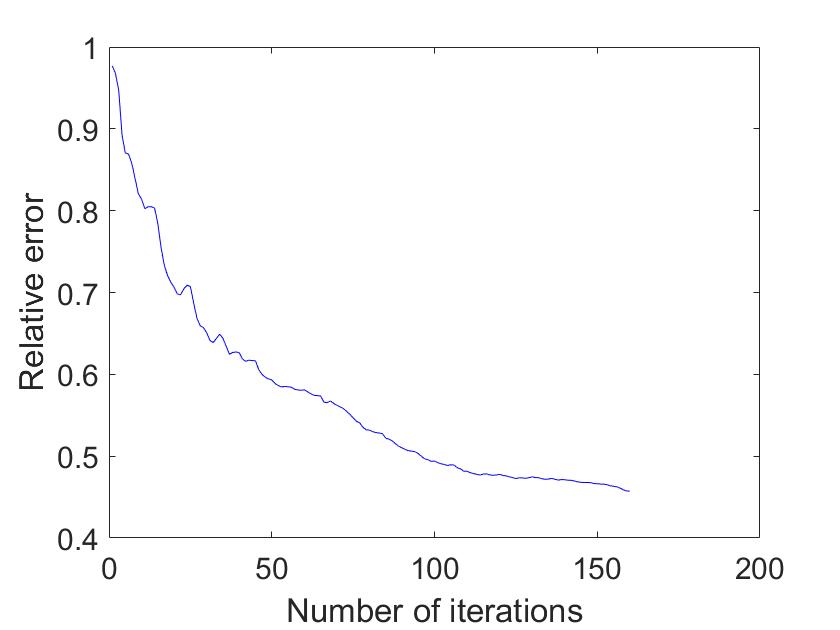} &
\includegraphics[scale=0.12]{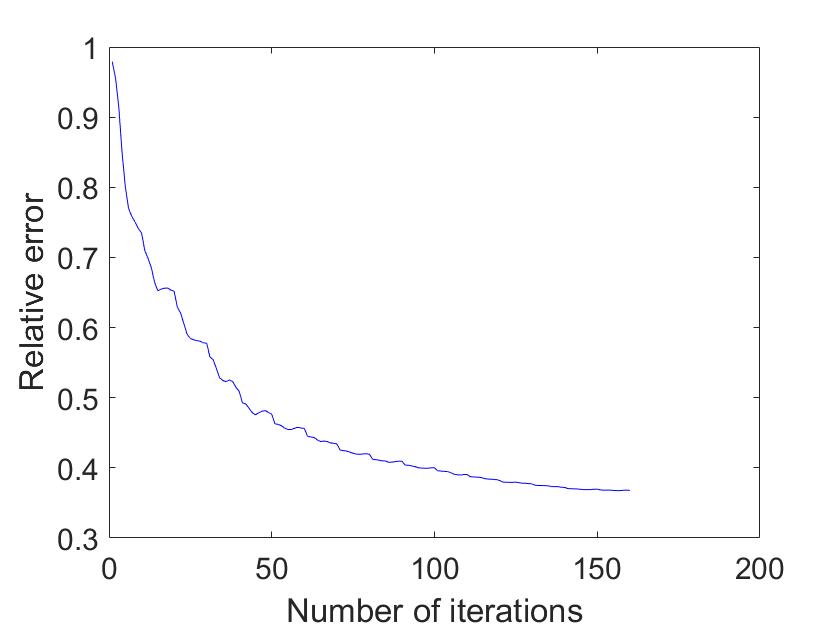} \\
(d) $e_{q_\lambda}$ & (e) $e_{q_\mu}$ & (f) $e_{q_\rho}$
\end{tabular}
\caption{Example 3: the reconstruction of perturbed parameters.}
\label{fig6}
\end{figure}

{\bf Example 4.} We use the measurements generated by the plane pressure incident wave with one fixed direction $d=(1,0)^\top$ (i.e. $M=1$). In this case the number of iterations at each frequency is set as $L=50$. The reconstruction results and relative errors are shown in Fig.\ref{fig7}.

\begin{figure}[htb]
\centering
\begin{tabular}{ccc}
\includegraphics[scale=0.12]{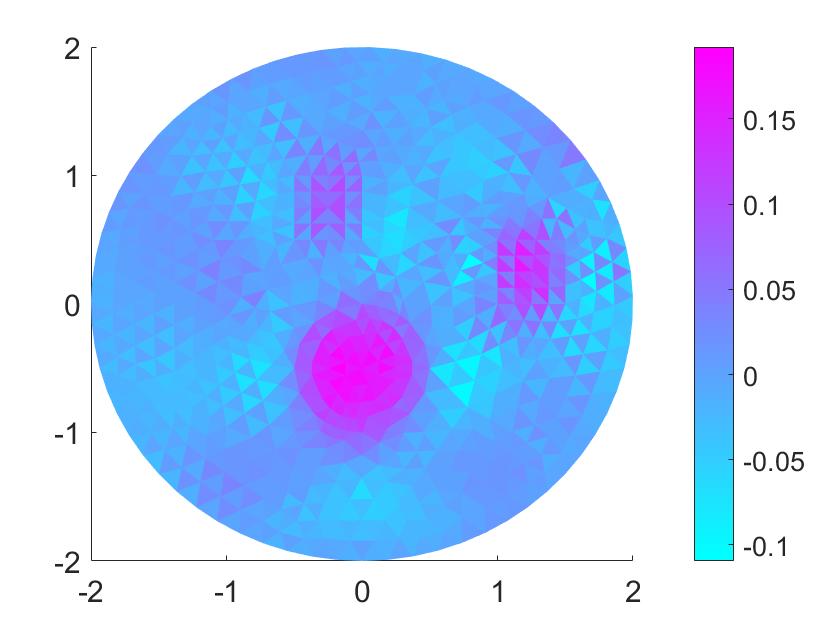} &
\includegraphics[scale=0.12]{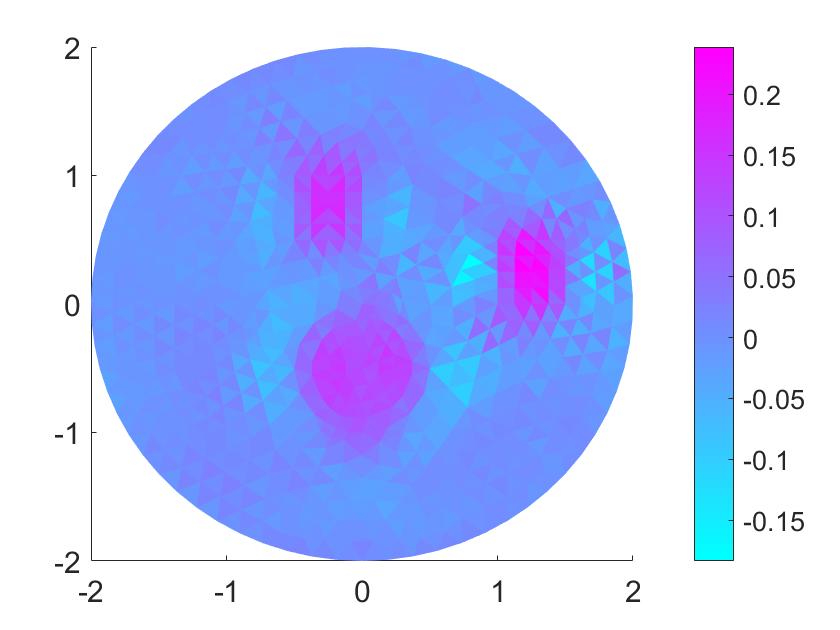} &
\includegraphics[scale=0.12]{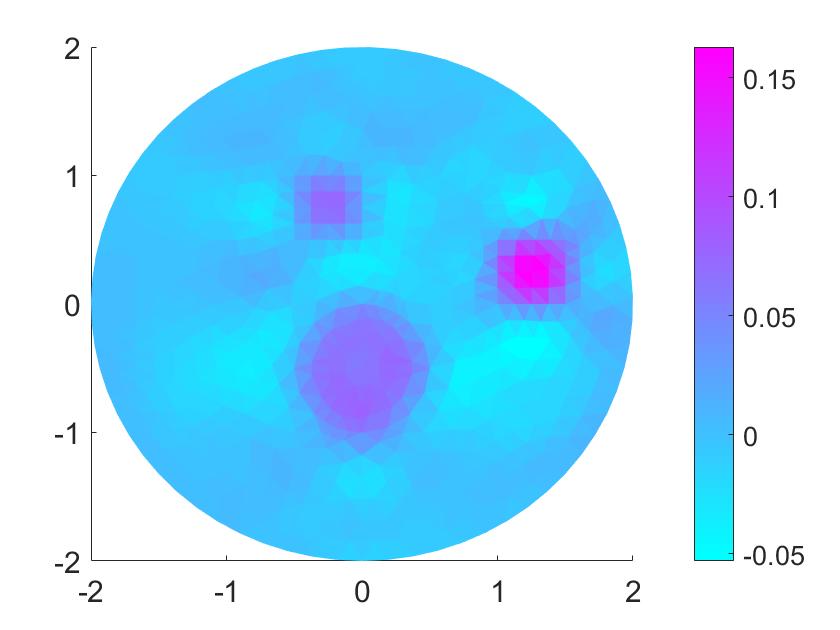} \\
(a) $q_\lambda$ & (b) $q_\mu$ & (c) $q_\rho$ \\
\includegraphics[scale=0.12]{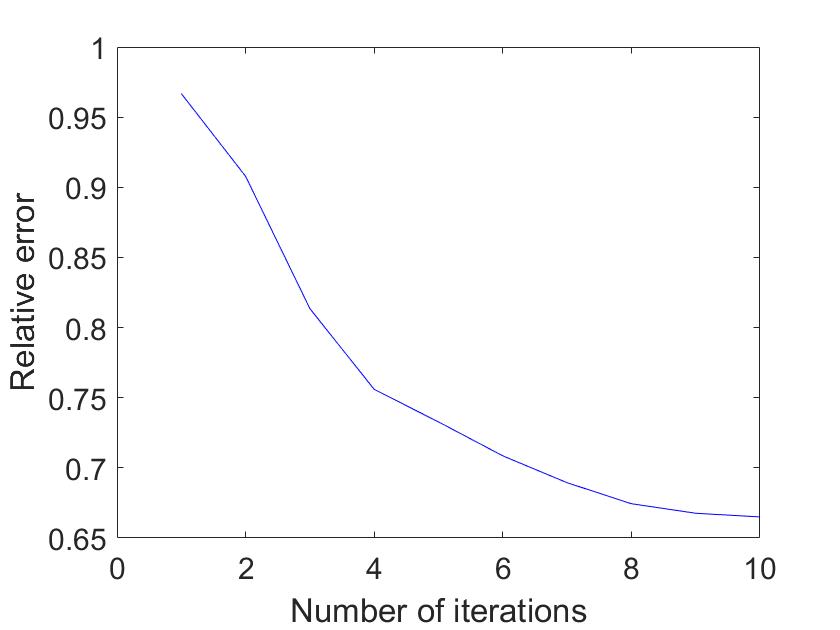} &
\includegraphics[scale=0.12]{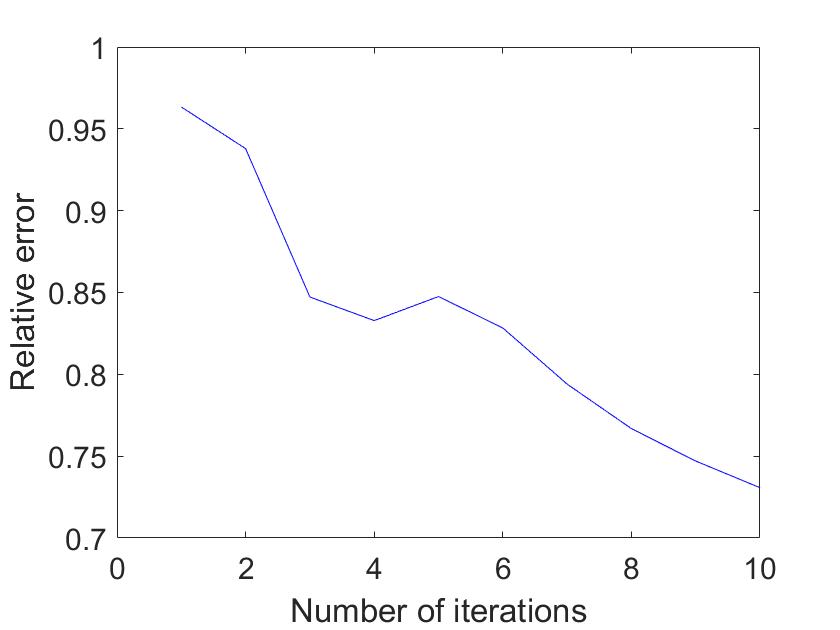} &
\includegraphics[scale=0.12]{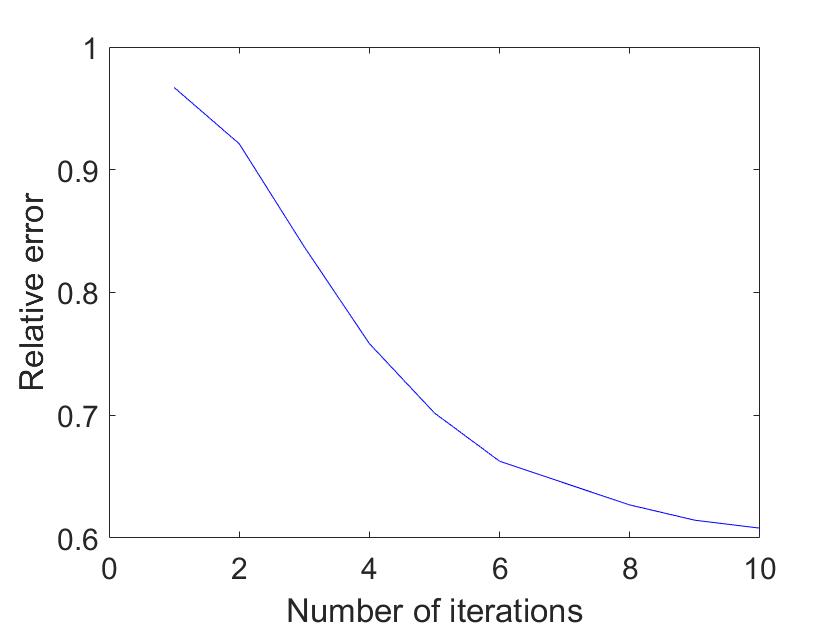} \\
(d) $e_{q_\lambda}$ & (e) $e_{q_\mu}$ & (f) $e_{q_\rho}$
\end{tabular}
\caption{Example 4: the reconstruction of perturbed parameters.}
\label{fig7}
\end{figure}

\begin{table}[htb]
\caption{\textcolor{rot1}{Final reconstruction errors of $q_\rho$ for Examples 5-6.}}
\centering
\begin{tabular}{|c|c|c|c|c|c|c|c|c|}
\hline
Figure & \ref{fig9} & \ref{fig10} & \ref{fig11}(a) & \ref{fig11}(b) & \ref{fig12}(a) & \ref{fig12}(b) & \ref{fig12}(c) & \ref{fig13} \\
\hline
Error of $q_\rho$ & 0.035 & 0.24 & 0.80 & 0.78 & 0.14 & 0.65 & 0.59 & 0.46  \\
\hline
\end{tabular}
\label{Table2}
\end{table}

{\bf Example 5.} \textcolor{rot1}{In this example}, we consider the special case discussed in section \ref{sec:4}. The exact value of $q_\rho$ is given by
\ben
q_\rho &=& 0.3(1-3x_1)^2\exp(-9x_1^2-(3x_2+1)^2) -(0.6x_1-27x_1^3-3^5x_2^5)\exp(-9x_1^2-9x_2^2)\\
&\quad& -0.03\exp(-(3x_1+1)^2-9x_2^2)
\enn
see Fig. \ref{fig8}. Choose $\alpha=0.01$ \textcolor{rot1}{and $\omega_{max}=11$}. The reconstructed $q_\rho$ and relative errors from multi-frequency measurements with plane pressure incident wave are shown in Fig.\ref{fig9} and Fig.\ref{fig10}.

\begin{figure}[htb]
\centering
\includegraphics[scale=0.3]{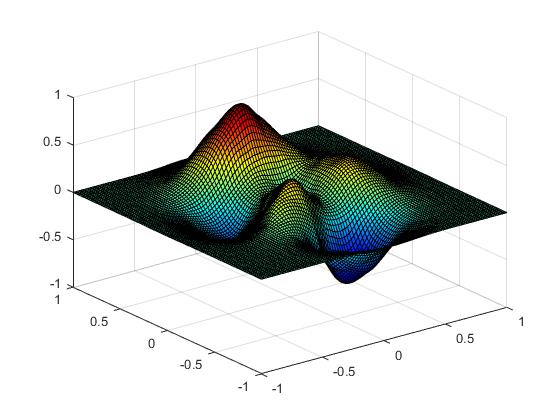}
\caption{Example 5: the exact value of $q_\rho$.}
\label{fig8}
\end{figure}

\begin{figure}[htb]
\centering
\begin{tabular}{cc}
\includegraphics[scale=0.15]{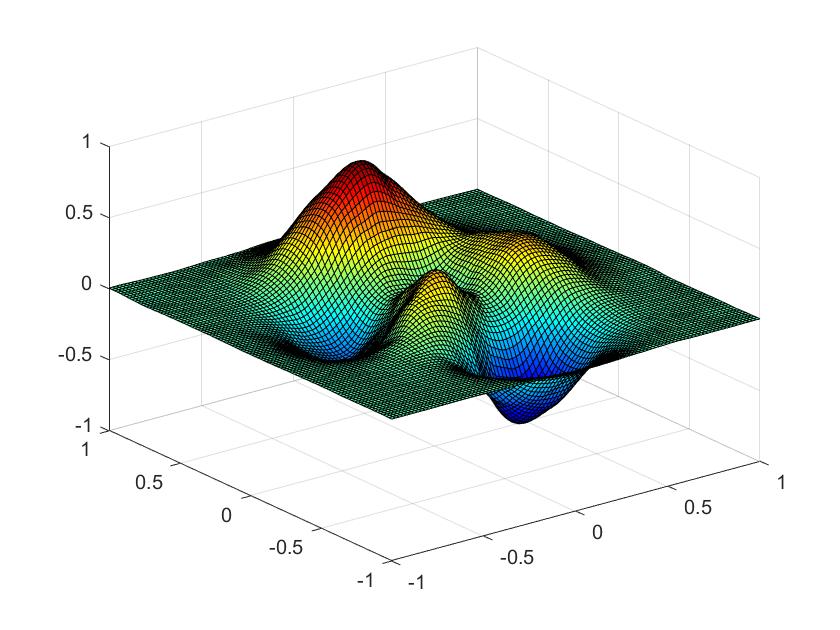} &
\includegraphics[scale=0.15]{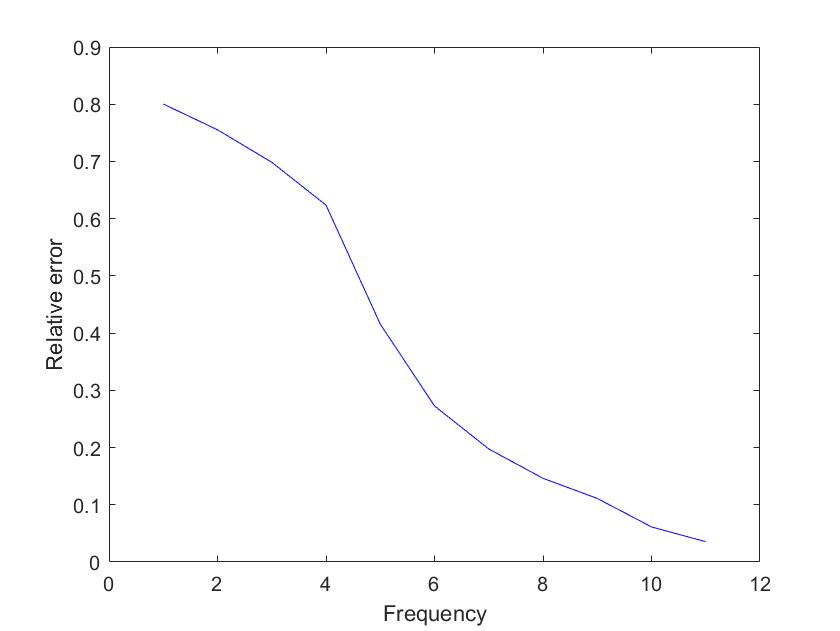} \\
(a) $q_\rho$ & (b) $e_{q_\rho}$
\end{tabular}
\caption{Example 5: the reconstruction of $q_\rho$ and relative errors from original measurements.}
\label{fig9}
\end{figure}

\begin{figure}[htb]
\centering
\begin{tabular}{cc}
\includegraphics[scale=0.15]{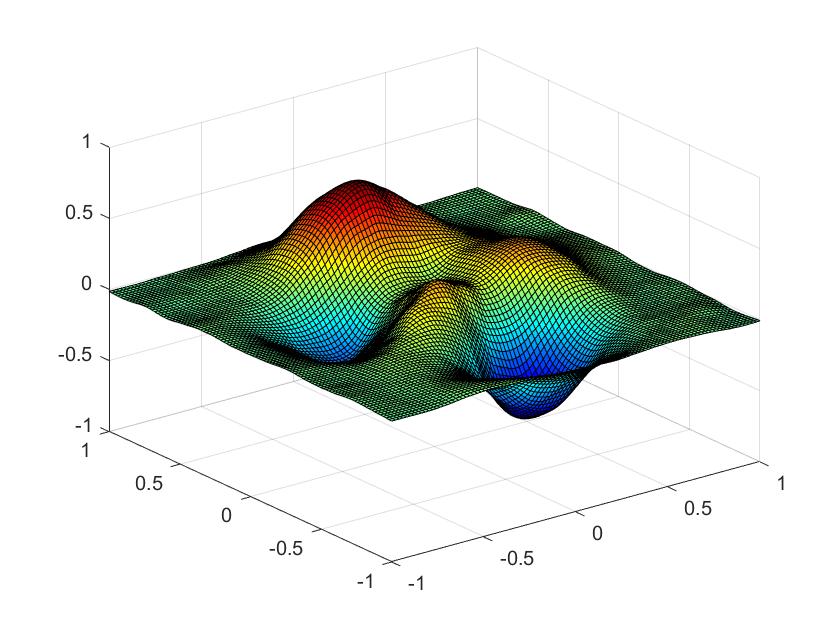} &
\includegraphics[scale=0.15]{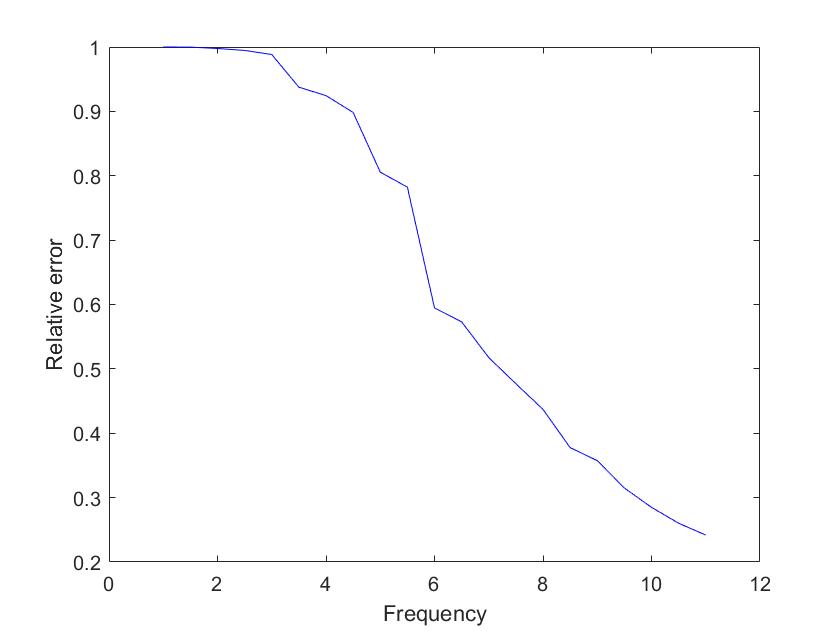} \\
(a) $q_\rho$ & (b) $e_{q_\rho}$
\end{tabular}
\caption{Example 5: the reconstruction of $q_\rho$ and relative errors from phaseless measurements.}
\label{fig10}
\end{figure}

\textcolor{rot1}{{\bf Example 6.} Finally, as a comparison, we consider the reconstruction of mass density only from data at a fixed frequency. For small frequency, it can be seen from Fig.\ref{fig11} that the reconstruction results are extremely bad no matter we have only one or multiple directions of incident wave. But for high frequency, see Fig.\ref{fig12}(a), we still can have good reconstruction if we have multiple incident waves. Once we only have one fixed incident wave with direction $d=(0,1)^\top$, we can not obtain good reconstruction result by increasing $L$, see Fig.\ref{fig12}(b,c). However, by increasing the number of frequency ($N=11$), we still can reconstruct some information of $q_\rho$, see Fig.\ref{fig13} and the results in Example 4. This further indicate the advantages to take multi-frequency data.}

\begin{figure}[htb]
\centering
\begin{tabular}{cc}
\includegraphics[scale=0.15]{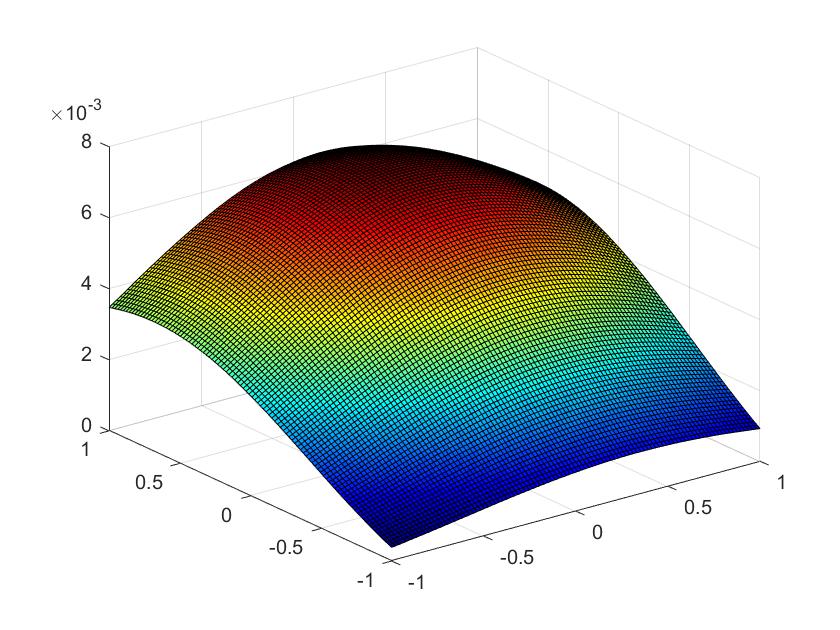} &
\includegraphics[scale=0.15]{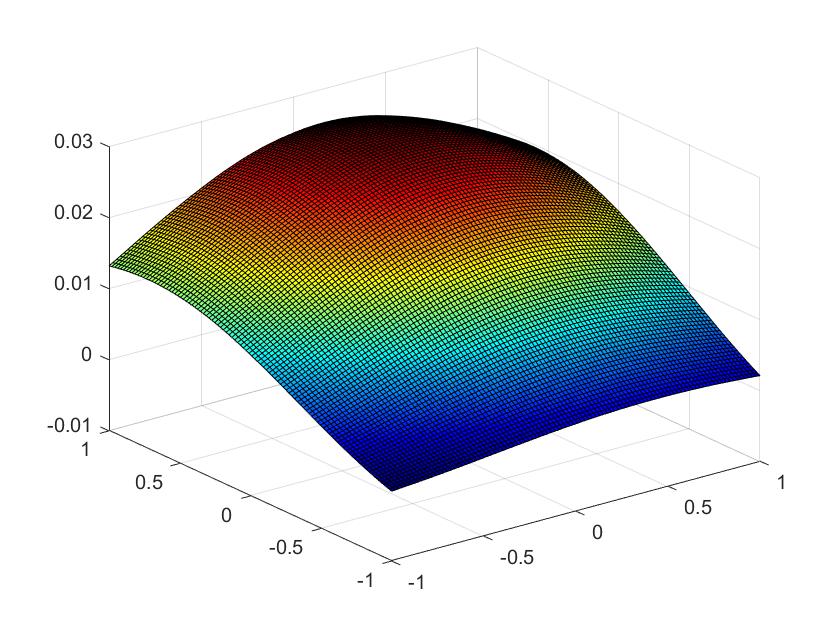} \\
(a) $M=1$ & (b) $M=16$
\end{tabular}
\caption{Example 6: the reconstruction of $q_\rho$ from data at $k=1$.}
\label{fig11}
\end{figure}

\begin{figure}[htb]
\centering
\includegraphics[scale=0.15]{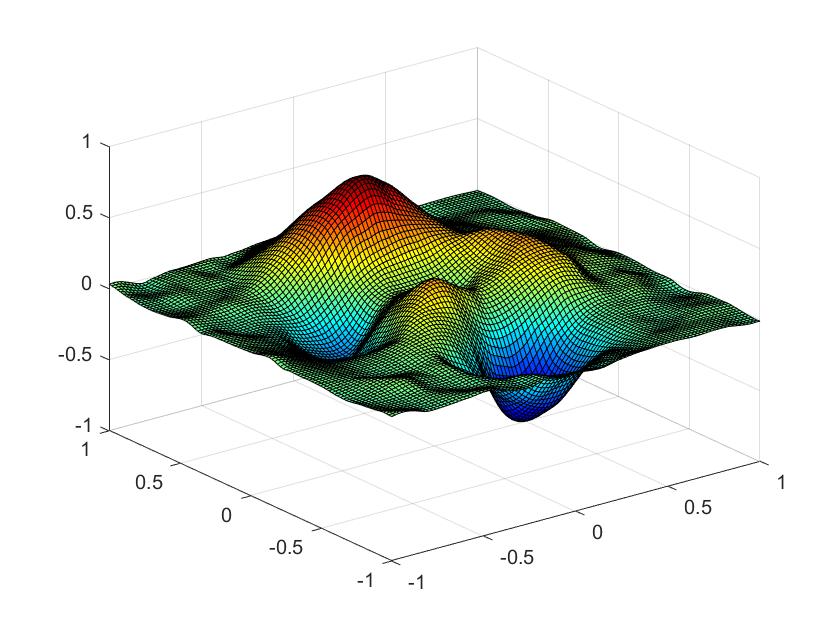}\\
(a) $M=16$, $L=10$\\
\begin{tabular}{cc}
\includegraphics[scale=0.15]{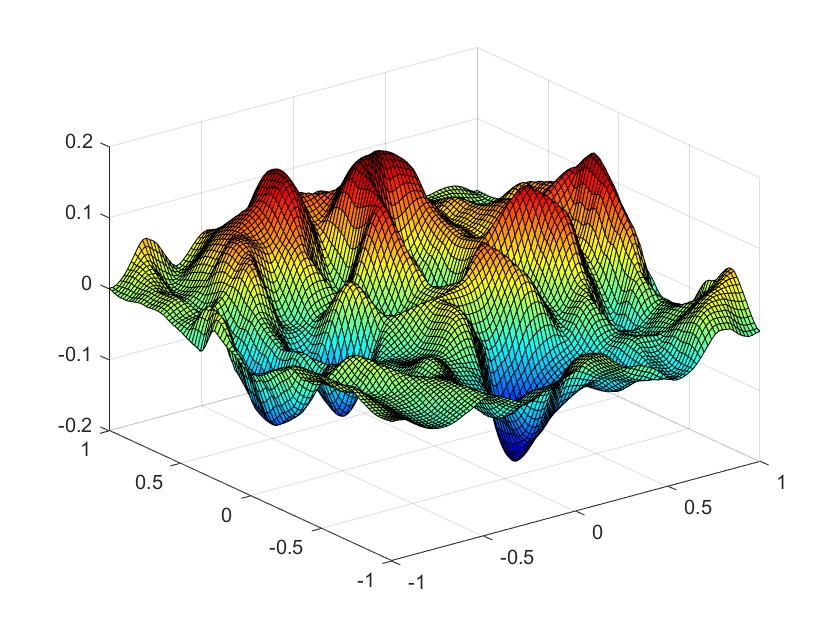} &
\includegraphics[scale=0.15]{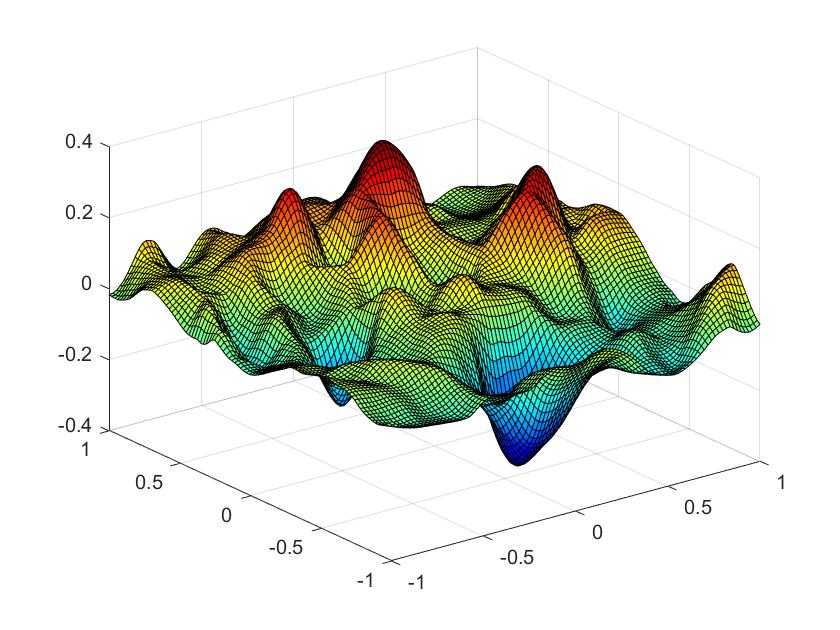} \\
(b) $M=1$, $L=10$ & (c) $M=1$, $L=100$
\end{tabular}
\caption{Example 6: the reconstruction of $q_\rho$ from data at $k=11$.}
\label{fig12}
\end{figure}

\begin{figure}[htb]
\centering
\includegraphics[scale=0.2]{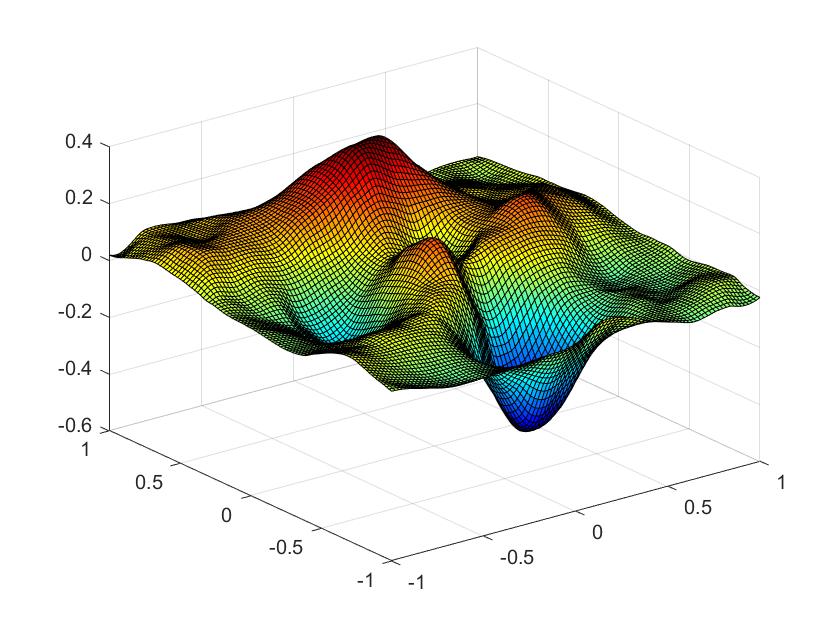}
\caption{Example 6: the reconstruction of $q_\rho$ from multi-frequency data with a fixed incident direction.}
\label{fig13}
\end{figure}

We conclude from the above numerical tests that satisfactory reconstructions are obtained through the proposed Landweber iterative algorithms.

\section*{Acknowledgments}
The work of G. Bao is supported in part by an NSFC Innovative Group Fund (No.11621101), an Integrated Project of the Major Research Plan of NSFC (No. 91630309), and an NSFC A3 Project (No. 11421110002). The work of F. Zeng is supported by the NSFC grant (No. 11501063, No. 11771068), the Chongqing Research Program of Basic Research and Frontier Technology (No. CSTC2017JCYJAX0294) and the Fundamental Research Funds for the Central University (No. 106112016CDJXY100004). The authors also would like to thank Prof. Peijun Li for his suggestions on this work.

\end{document}